\documentclass[11pt,letterpaper]{amsart}
\usepackage[margin=1in]{geometry}
\usepackage{graphicx} 
\usepackage{amsmath}
\usepackage{numprint}
\usepackage{comment}
\usepackage{amsthm}
\usepackage{amssymb}
\usepackage[shortlabels]{enumitem}
\usepackage{quiver}
\usepackage{soul}
\usepackage{tikz}
\usepackage{tikz-3dplot}
\usetikzlibrary{patterns}
\usepackage{mathtools}
\newtheorem{theorem}{Theorem}[section] 
\newtheorem{proposition}[theorem]{Proposition}
\newtheorem{question}[theorem]{Question}
\newtheorem{corollary}[theorem]{Corollary}

\theoremstyle{definition}
\newtheorem{definition}[theorem]{Definition}
\newtheorem{example}[theorem]{Example}

\newtheorem{remark}[theorem]{Remark}

\usepackage{tikz-cd}
\usepackage[normalem]{ulem}
\usepackage{biblatex}
\usepackage{amsfonts}
\usepackage{hyperref}

\newcommand{\il}[1]{\textit{#1}}
\newcommand{\g}[1]{\textbf{#1}}

\usepackage{xcolor}

\usepackage{tikz}
\usetikzlibrary{fit}

\newtheorem{manualtheoreminner}{Theorem}
\newenvironment{manualtheorem}[1]{%
  \IfBlankTF{#1}
    {\renewcommand{\themanualtheoreminner}{\unskip}}
    {\renewcommand\themanualtheoreminner{#1}}%
  \manualtheoreminner
}{\endmanualtheoreminner}

\title{Poset partitions and the Combinatorics of the $\textbf{cd}$-index}

\author[F. Caster]{Felipe Caster}
\address{Facultad de Matem\'aticas, Pontificia Universidad Cat\'olica, Santiago, Chile}
\email{felipe.caster@uc.cl}
\author[D. Guyer]{Dan Guyer}
\address{Department of Mathematics, University of Washington, Seattle, WA, United States}
\email{dguyer@uw.edu}
\author[J.A. Samper]{Jos\'e Alejandro\ Samper}
\address{Facultad de Matem\'aticas, Pontificia Universidad Cat\'olica, Santiago, Chile}
\email{jsamper@uc.cl}

\thanks{FC is partially supported by ANID Becas/Magister Nacional 22250673. DG is partially\textsl{} supported by NSF grant  DMS-2246399. JAS is partially supported by ANID FONDECYT Iniciaci\'on grant \#11221076.}

\begin{document}

		\begin{abstract}
		We introduce a new class of Eulerian posets, called \il{S}-partitionable posets, which have a non-negative \g{cd}-index. These posets are a generalization of \il{S}-shellable complexes introduced by Stanley in 1994. We prove that \il{S}-partitionable posets have a non-negative \g{cd}-index via a recursive formula. Then, we introduce a semi-Eulerian version of \il{S}-partitionable posets, which we call \il{SE}-partitionable posets. We show that \il{SE}-partitionable posets also have a non-negative semi-Eulerian \g{cd}-index as defined by Juhnke-Kubitzke, Samper and Venturello in 2024. 
        \end{abstract}
\maketitle
\section{Introduction}  

The theory of partially ordered sets, or posets, is a central area in enumerative and algebraic combinatorics that was brought to the spotlight by Rota in his influential work ~\cite{Rota_1964} as highlighted by Stanley~\cite{Stanley_2021}. One of the key insights of Rota is that there is a deep connection between enumeration and topology. 

The order complex of a poset is the simplicial complex whose vertices are elements and whose faces are chains. In many instances, the combinatorial invariants of the poset, like the M\"obius function, are equivalent to topological data, such as the Euler characteristic. A topic that has received considerable attention is that of relating combinatorial and topological invariants of both families. For instance, what can be said about the combinatorial structure of a poset whose order complex is homeomorphic to a sphere, a manifold, or some other space of interest? The celebrated $g$-theorem \cite{Billera_1981,adiprasito_2019,papadakis_2020,Karu_2023} gives a numerical classification of all possible $f$-vectors of triangulated spheres, which are numbers enumerating elements of the face poset by rank. 

From a combinatorial point of view, spheres and related spaces are difficult to understand; thus, a frequent approach is to consider a wider class of objects that is defined in terms of the combinatorial structure of the poset. For example, all spheres are Gorenstein* posets, and all Gorenstein* posets are Eulerian. Both classes are rich in examples and extract key properties of spheres for some of the results mentioned below. Roughly, Gorenstein* posets (over a field $\mathbb{F}$) are posets such that homology groups cannot distinguish them from a sphere, while Eulerian posets are posets such that the Euler characteristic cannot distinguish them from spheres. 

An interesting combinatorial invariant for any poset concerns the counts of chains of different types. If $\Omega$ is a bounded, graded poset of rank $d+1$, and $K\subseteq [d]$, let $f_K(\Omega)$ be the number of chains of $\Omega$ such that the set of ranks of its elements is $K$. The \textit{flag $f$-vector} of $\Omega$ is the vector $(f_K(\Omega))_{K\subseteq[d]}$. 

For Eulerian posets, the flag $f$-vector satisfies a collection of linear relations called \textit{generalized Dehn--Sommerville equations}. Inspired by the classical Dehn--Sommerville equations, Bayer and Billera \cite{Billera_1983,Billera_1985} determined all the linear relations satisfied by the collection of all Eulerian posets of rank $d+1$ and showed that the affine space has an affine basis formed by $F_{d+1}-1$ polytopes, where $F_n$ denotes the $n$-th Fibonacci number. 

The \g{cd}-index, discovered by Jonathan Fine and developed by Bayer-Klapper~\cite{Bayer1991}, provides an optimal codification of the flag $f$-vector. The idea is to express the entries of the flag $f$-vector as a non-negative linear combination of the coefficients of a degree $d$ homogeneous noncommutative polynomial in variables \g{c} and $\g{d}$ of degrees $1$ and $2$ respectively.

By the nature of the combination, the coefficients of the $\g{cd}$-index are smaller than the coefficients of the $f$-vector whenever they are positive. Unfortunately, there are Eulerian posets whose $\g{cd}$-index has negative coefficients, so the meaning of these coefficients is unclear in general. Yet, the $\g{cd}$-index is non-negative for many posets of interest. In 1994, Stanley \cite{Stanley1994} inductively proved the non-negativity of the \g{cd}-index of all \il{S}-shellable spheres, and hence proved that the \g{cd}-index of all polytopes is non-negative. In the same article, he conjectured that all Gorenstein* posets have a non-negative \g{cd}-index and proved it for all simplicial Gorenstein* posets. Later, in 1999, Stanley~\cite{Stanley_1999} wrote, ``In general, the \g{cd}-index is a highly intractable object. It would be of great interest to find a natural algebraic or geometric description of the \g{cd}-index." The conjecture was settled by Karu~\cite{Karu_2006} using cohomological methods and Yuzvinsky's theory of sheaves over posets to realize the \g{cd} coefficients as dimensions of different vector spaces. 

There have been several notions defining generalizations of the \g{cd}-index. Ehrenborg~\cite{ehrenborg01} extended the \g{cd}-index theory to the class of $k$-Eulerian posets. Ehrenborg, Goresky, and Readdy~\cite{Ehrenborg_2015} also defined a \g{cd}-index for Whitney stratified manifolds. More recently, Juhnke-Kubitzke, Samper, and Venturello~\cite{semi} defined the \g{cd}-index for semi-Eulerian posets, which include all regular CW-manifolds. They give several reasons why this definition is natural. In particular, they show that the
\g{cd}-index of any Buchsbaum simplicial poset is well defined and non-negative, a class that includes all posets
whose order complex is a homology manifold. They also outline several directions for further
investigation. 

In most cases above, the proofs of positivity do not provide a combinatorial meaning for the coefficients of the \g{cd}-index for any class of Eulerian posets. In fact, in a recent survey, Bayer~\cite{Bayer_2021} states the following: ``There are two main issues for research on \g{cd}-indices. One is the question of the non-negativity of the coefficients, or, more generally, inequalities on the \g{cd}-index, for Eulerian posets or for particular subclasses. The other (related) issue is the combinatorial interpretation of the coefficients.'' Our goal in this article is to make progress on both of these aspects by developing some new combinatorial tools to prove \g{cd}-positivity. The main idea is to extend the notion of \il{S}-shellability to a notion of \il{S}-partitionability in a way similar to how partitionability of simplicial complexes extends shellability for the computation of h-vectors. 

In the theory of face enumeration of simplicial complexes, the $h$-vector is an integer vector that indirectly counts faces of different dimensions. Some entries of this vector may be negative, but there are remarkable families of simplicial complexes whose entries are non-negative. Shellability and partitionability are classical combinatorial tools that imply the non-negativity of the $h$-vector. A shelling builds a complex inductively in an organized way, and a partition extracts the essential features of a shelling that yield a combinatorial interpretation for the entries of the $h$-vector. Our main goal is to study posets that are currently out of reach by means of algebraic methods like those used by Karu and provide tools to explore the semi-Eulerian case further. 

The $\g{cd}$-index is a generalization of the $h$-vector for Eulerian posets: it has integral entries encoding all linear relations of the flag $f$-vector, and while some of the entries can be negative, there are many classes of Eulerian posets that have non-negative coefficients. Stanley's notion of \il{S}-shellability adapts classical shellability to the context of the \g{cd}-index to prove that all \il{S}-shellable posets have non-negative \g{cd}-index. Unlike the combinatorial interpretation of $h$-vector as enumerating facets by sizes of restriction sets, the recursive proof by Stanley does not yield an easy way to interpret the coefficients. \il{S}-shellability was inspired by line shellings and Lee~\cite{Lee2010} showed how to partition the face poset of a polytope to enumerate the coefficients of the $\g{cd}$-index in terms of an \il{S}-shelling of its dual. 

Inspired by Lee's work, we define the notion of \il{S}-partitionability: a way to partition the elements of an Eulerian poset that extracts the key features guaranteeing the non-negativity of the $\g{cd}$-index of an \il{S}-shellable Eulerian poset. An \il{S}-partition of an Eulerian poset $\Omega$ is a partition of the elements of $\Omega\setminus \{\hat{1}\}$, with one partition class per coatom that satisfies certain recursive properties. All \il{S}-shellable spheres and face posets of partitionable Eulerian simplicial complexes are \il{S}-partitionable. Moreover, our proof techniques remove the geometric assumptions of Lee~\cite{Lee2010} and apply the recursion(s) in a way that each coatom contributes a non-negative sum of \g{cd}-words to the \g{cd}-index of $\Omega$. One immediate advantage of \il{S}-partitionability is that it imposes fewer topological restrictions than shellability: indeed \il{S}-shellable posets have order complexes homeomorphic to spheres, whereas we are able to provide several examples of other homology manifolds admitting \il{S}-partitions. Our first main theorem is as follows:

\begin{manualtheorem}{A}
    \textit{The \g{cd}-index of an \il{S}-partitionable poset can be recursively computed in terms of its \il{S}-partitionable parts of smaller rank, and is consequently non-negative.} 
\end{manualtheorem}

The proof of the theorem mimics Lee's proof for the decomposition of the $\g{cd}$-index of a polytope, but it replaces several of the geometric notions with suitable combinatorial analogs. We further obtain a partition of the face lattice of the order complex that counts the \g{cd}-index of the initial poset, but the decomposition depends on some non-explicit steps that are guaranteed to exist by the generalized Dehn--Sommerville equations. 

For semi-Eulerian complexes, the \il{S}-partitionability is not enough, and the modification of the flag $f$-vector that is needed to define the $\g{cd}$-index poses an additional enumerative challenge. We propose a relaxation of \il{S}-partitionability, called \il{SE}-partitionability, in which the recursive part is slightly modified to accommodate the corrections made to the flag $f$-vector. Every \il{S}-partitionable poset is \il{SE}-partitionable, and there are several examples of manifolds that are not Eulerian but are \il{SE}-partitionable. Our second main theorem reads as follows: 

\begin{manualtheorem}{B}
    \textit{The \g{cd}-index of an \il{SE}-partitionable poset can be recursively computed in terms of its \il{S}-partitionable parts of smaller rank, and is consequently non-negative.} 
\end{manualtheorem}

Again, the proof is an adaptation of the case of \il{S}-partitionable posets with the appropriate modifications. Several interesting examples follow. For instance, any simplicial semi-Eulerian complex admitting a partition is \il{SE}-partitionable. Furthermore, all triangulations of two dimensional simplicial semi-Eulerian pseudomanifolds are \il{SE}-partitionable as well. Our results provide several heuristic reasons to expect positivity, and we hope they will be useful in establishing positivity (or reasonable lower bounds) for the \g{cd}-index of regular CW-homology manifolds.

This article is organized as follows. Section 2 presents definitions and preliminary results about posets, $\g{cd}$-indices, and other related notions. Section 3 is devoted to developing the theory \il{S}-partitionable posets and proving the non-negativity of their \g{cd}-indices. Section $4$ discusses the variations needed to define \il{SE}-partitionable posets. Section $5$ discusses some open problems, future directions, and possible connections to existing literature.

\section{Preliminaries}
We begin by introducing standard poset notation that will be used throughout the paper, and reviewing some background on the \g{cd}-index for Eulerian and semi-Eulerian posets. For more details see \cite{Stanley_ec1,Stanley1994,Bayer_2021,semi}. For content related to polytopes, we refer to Ziegler~\cite{Ziegler_1995}, and for content related to combinatorial decompositions of simplicial complexes, we refer to Hachimori~\cite{Hachimori_2000}.

Throughout, we work with finite graded posets $\Omega$ having a unique minimum element $\hat{0}$ of rank $0$ and a unique maximum element $\hat{1}$ of rank $d+1$. If $\Omega$ is the face poset of a regular CW complex, we will usually refer to the face poset and the regular CW complex itself by $\Omega$.

\begin{definition}
For a poset $\Omega$, the \textit{order complex} $\Delta(\Omega)$ is the simplicial complex given by all chains in $\Omega\setminus \{\hat{0},\hat{1}\}$.
\end{definition}

For a regular CW complex $\Omega$, the geometric realization of its order complex $\Delta(\Omega)$ is homeomorphic to the barycentric subdivision of $\Omega$, and hence to $\Omega$ itself~\cite[pg.~80]{Lundell_1969}. Thus, the face poset of $\Omega$ encodes all the essential combinatorial and topological information of the complex. In particular, we are interested in studying combinatorial invariants of special classes of regular CW complexes, such as regular CW homology spheres and homology manifolds.

\begin{definition}
A graded poset $\Omega$ with $\hat{0}$ and $\hat{1}$ is called \textit{Eulerian} if every interval $[x,y]$ with $x<y$ has the same number of elements of odd rank as elements of even rank. Equivalently, a graded poset with $\hat{0}$ and $\hat{1}$ is \textit{Eulerian} if its Möbius function satisfies $\mu([x,y]) = (-1)^{\,r(y)-r(x)}$ for all intervals $[x,y]$.
\end{definition}

Typical examples of Eulerian posets include face lattices of convex polytopes and, more generally, the face posets of regular homology spheres. Face posets of odd-dimensional manifolds are also Eulerian, whereas even-dimensional manifolds need not be—these give rise to the broader class of semi-Eulerian posets.

\begin{definition}
A graded poset $\Omega$ with $\hat{0}$ and $\hat{1}$ is called \textit{semi-Eulerian} if every interval $[x,y]\neq [\hat{0},\hat{1}]$ with $x<y$ has the same number of elements of odd rank as elements of even rank. Equivalently, a graded poset with $\hat{0}$ and $\hat{1}$ is \textit{semi-Eulerian} if the Möbius function satisfies $\mu([x,y]) = (-1)^{\,r(y)-r(x)}$ for all intervals $[x,y]\neq [\hat{0},\hat{1}]$.
\end{definition}

Thus, semi-Eulerian posets are bounded graded posets whose local Euler characteristic agrees with that of a sphere of the appropriate dimension, though not necessarily globally. Examples of semi-Eulerian posets include face posets of regular CW decompositions of connected topological manifolds or, more generally, of homology manifolds.  

\begin{definition}
   A \textit{pseudomanifold} is a $d$-dimensional simplicial complex $\Delta$ that satisfies:  
(1) $\Delta$ is pure;  
(2) every ridge (codimension-one face) lies in exactly two facets;  
(3) $\Delta$ is strongly connected: Any two facets $\sigma_0$, $\sigma_k$ of $\Delta$ can be connected by a sequence of adjacent facets $\sigma_0,\sigma_1,\hdots,\sigma_k$, i.e, $\sigma_i\cap \sigma_{i+1}$ is a common $(d-1)$-dimensional face for $0\le i \le k-1$. In particular, triangulations of connected homology manifolds are pseudomanifolds.  
\end{definition}

Next, we discuss some definitions that will be heavily relied upon in the recursive decomposition of an \il{S}-partitionable poset and its \g{cd}-index. 

\begin{definition}
    We say that $\Omega$ is \textit{near-Eulerian} if it can be obtained from an Eulerian poset by deleting a single coatom.
\end{definition} 

Near-Eulerian posets arise naturally when building a recurrence for the \g{cd}-index as done in Stanley~\cite{Stanley1994}. To precisely describe parts of a poset, we introduce the following definitions.

\begin{definition}
For $p\in \Omega$, we write $\overline{p} \coloneq [\hat{0},p]$ for the (down) \textit{closure} of $p$, and for a collection $\mathcal{C}\subseteq \Omega$, we define $
\overline{\mathcal{C}} \coloneq \bigcup_{p\in \mathcal{C}} \overline{p}.$
\end{definition}

\begin{definition}
    For a near-Eulerian poset $\Omega$, the closure of the union of all $y$ with $[y,\hat{1}]$ a three-element chain is called the \textit{boundary} of $\Omega$.
\end{definition}

\begin{remark}
If $\Omega$ is near-Eulerian, then the associated Eulerian poset can be reconstructed by adjoining a new coatom $\tau$ that covers every $y$ with $[y,\hat{1}]$ a three-element chain. This operation is called the \textit{semisuspension} of $\Omega$, denoted $\tilde{\Sigma}\Omega$.
\end{remark}
For a graded poset of $\Omega$ of rank $d+1$ with rank function $r$ and $K\subseteq [d]$, the $K$-flag number $f_K(\Omega)$ denotes the number of chains $G = \{x_1 \prec x_2 \prec \hdots \prec x_k:x_i\in \Omega\}$ whose \textit{rank set} is $r(G) = \{r(x_i) : 1 \le i \le k\}=K$.
The \textit{flag $f$-vector} $(f_K)_{K\subseteq [d]}$ allows one to encode much combinatorial information of $\Omega$. However, one often works with the \textit{flag $h$-vector} $(h_K)_{K\subseteq [d]}$, defined by the invertible transformation
$$ h_K(\Omega) = \sum_{T\subseteq K} (-1)^{|K\setminus T|} f_T(\Omega).$$

An alternative encoding of the flag $f$-vector and flag $h$-vector uses polynomials in non-commuting variables $\g{a,b}$ in $\mathbb{Z}\langle \g{a},\g{b}\rangle$. The \textit{\g{ab}-polynomial} $\Psi_\Omega(\g{a},\g{b})$ and \emph{chain polynomial} $\Upsilon_\Omega(\g{a},\g{b})$ are given by
$$
\Psi_\Omega(\g{a},\g{b}) = \sum_{K\subseteq [d]} h_K(\Omega)\, u_K, 
\qquad
\Upsilon_\Omega(\g{a},\g{b}) = \sum_{K\subseteq [d]} f_K(\Omega)\, u_K,
$$
where $u_K = u_1 \cdots u_n$ with $u_i = \g{a}$ if $i\notin K$ and $u_i = \g{b}$ if $i\in K$. These polynomials are related by
$$
\Upsilon_\Omega(\g{a}-\g{b},\g{b}) = \Psi_\Omega(\g{a},\g{b}).
$$

The \il{generalized Dehn--Sommerville equations}, introduced by Bayer-Billera~\cite{Billera_1983}~\cite{Billera_1985}, describe the affine span of all flag $f$-vectors of Eulerian posets of rank $d+1$. More concretely, for any $K\subseteq  [d]$, we have
\[
\sum_{j=i+1}^{k-1} (-1)^{\,j-i-1}\, f_{K\cup\{j\}}(\Omega) 
= \bigl(1-(-1)^{\,k-i-1}\bigr)\, f_K(\Omega),
\]
where $i \le k-2$, $i,k \in K \cup \{0,d+1\}$, and $K \cap \{i+1,\dots,k-1\} = \varnothing$. 

\begin{definition}
Let $\Omega$ be a graded poset of rank $d+1$. The \emph{\g{cd}-index} of $\Omega$ is the unique polynomial $\Phi_\Omega(\g{c},\g{d})$ with $\g{c}=\g{a}+\g{b}$ and $\g{d}=\g{ab}+\g{ba}$ such that
\[
\Phi_\Omega(\g{a}+\g{b},\g{ab}+\g{ba}) = \Psi_\Omega(\g{a},\g{b}),
\]
whenever such a polynomial exists. 
\end{definition}
We may often abuse notation and write $\Phi(\Omega)$ as shorthand for $\Phi_\Omega(\g{c},\g{d})$. If $\Omega$ has a \g{cd}-index, then $\Phi_\Omega(\g{c},\g{d})$ is unique and homogeneous of degree $d$ in $\mathbb{Z}\langle \g{c},\g{d}\rangle$, assigning degree $1$ to \g{c} and degree $2$ to \g{d}. The number of monomials (\g{cd}-words) of degree $d$ equals the $(d+1)^{\text{th}}$ Fibonacci number $F_{d+1}$ defined with initial conditions $F_1=F_2=1$.

More precisely, we consider the unique graded map $$\Phi\colon \mathbb Z\langle \g c,\g d\rangle \to \mathbb Z\langle \g a, \g b\rangle$$ sending $\g c\mapsto \g a+\g b$ and $\g d\mapsto \g{ab}+\g{ba}$. We say that a non-commmutative polynomial in variables $\g{a}$ and $\g{b}$ has a $\g{cd}$-index if it lies in the image of $\Phi$. Furthermore, a chain polynomial lies in the image if and only if its corresponding flag $f$-vector satisfies the generalized Dehn--Sommerville equations. In particular, the flag $f$-vector of every Eulerian poset satisfies these relations.
\begin{theorem}[\cite{Bayer1991}]\label{thm: BayerKlapper}

A graded poset has a \g{cd}-index if and only if its flag $f$-vector satisfies the generalized Dehn--Sommerville equations. In particular, every Eulerian poset has a \g{cd}-index.
\end{theorem}  

To define the $\g{cd}$-index for semi-Eulerian posets, we need to consider the following modification of the flag $f$-vector, given by Juhnke-Kubitzke, Samper and Venturello \cite{semi}, that generalize both the \g{ab}-polynomial and the chain polynomial. Furthermore, it coincides with the classical definitions in the Eulerian case and enables the definition of a \g{cd}-index for semi-Eulerian posets. This modification relies on a small shift using the Euler characteristic of $\Omega$, denoted as $\chi(\Omega)$. Additionally, we use $\mathbb{S}^{d-1}$ to denote the $(d-1)$-dimensional sphere.

\begin{definition}\label{def:semicd} Let $\Omega$ be a graded poset of rank $d+1$ with unique minimal and maximal elements $\hat 0$ and $\hat 1$. For $K\subseteq [d]$, let 

 $$f'_K(\Omega)=\begin{cases}
f_{\{d\}}(\Omega) + \left( \chi(\mathbb{S}^{d-1}) - \chi(\Omega) \right), & \text{if } K = \{d\}, \\
f_K(\Omega), & \text{otherwise}.
\end{cases}$$
The polynomial $
\chi'_\Omega(\g{a}, \g{b}) = \sum_{K \subseteq [d]} f_K'(\Omega) \, u_K$ is called the \textit{modified chain polynomial}. 
\end{definition}

\begin{theorem}[\cite{semi}]\label{th:semicd} Let $\Omega$ be a semi-Eulerian poset of rank $d+1$, The modified chain polynomial $\chi'_\Omega(\g{a},\g{b})$ has a \g{cd}-index, i.e, there exists a polynomial $\Phi_\Omega(\g{c},\g{d})$ in non-commuting variables $\g{c}=\g{a}+\g{b}$ and $\g{d}=\g{ab}+\g{ba}$ such that 
$$\Phi_\Omega(\g{a}+\g{b},\g{ab}+\g{ba})=\chi_\Omega'(\g{a}-\g{b},\g{b}).$$
\end{theorem}

Eulerian posets will be the focus of Section~\ref{sec:S_part} of this article. Meanwhile, Section~\ref{sec:SE_part} concerns {semi-Eulerian} posets.

\section{S-partitionable Posets}\label{sec:S_part}
In this section, we will define the class of Eulerian \il{S}-partitionable posets, prove that their \g{cd}-index is non-negative and show that they generalize both Stanley's \il{S}-shellable spheres and Eulerian partitionable simplicial complexes. To begin, we provide some examples illustrating the recursive nature of the partition(s) of interest that we utilize in the proof of the non-negativity of the \g{cd}-index of \il{S}-partitionable posets.

\subsection{A Motivating Example}
Before diving into the technical definition of \il{S}-partitionable posets, we begin with a motivating example to demonstrate how the \g{cd}-index shall be computed. We demonstrate this count for the (boundary of) polytope $Q$ in Figure~\ref{fig:Q_ex}.

\tdplotsetmaincoords{70}{61}
\begin{figure}[htbp]
\centering
\begin{minipage}[t]{0.48\textwidth}
\centering
\tdplotsetmaincoords{70}{61}
\begin{tikzpicture}
[tdplot_main_coords,scale=3.9]

\coordinate (A) at (0,0.05,0.03);
\coordinate (B) at (0.7,0,0);
\coordinate (C) at (0.4,0.8,0);


\coordinate (F) at (0.8,0.4,0.65);
\coordinate (G) at (0.2,0.17,0.5);
\coordinate (H) at (0.55,0.17,0.45);

\coordinate (D) at (0.35,0.33,-0.60);

\fill[gray!20] (F)--(G)--(H)--cycle;
\fill[gray!20] (A)--(B)--(H)--(G)--cycle;
\fill[gray!20] (B)--(C)--(F)--(H)--cycle;
\fill[gray!10] (A)--(B)--(D)--cycle;
\fill[gray!20] (B)--(C)--(D)--cycle;
\fill[gray!20] (C)--(A)--(D)--cycle;

\draw[thick] (F)--(G)--(H)--cycle;
\draw[thick] (A)--(B)--(C);
\draw[dashed] (A)--(C);
\draw[thick] (A)--(B)--(H)--(G)--cycle;
\draw[thick] (B)--(C)--(F)--(H)--cycle;
\draw[thick] (A)--(D)--(B);
\draw[thick] (B)--(D)--(C);
\draw[thick] (C)--(D)--(A);
\end{tikzpicture}
\caption{Polytope $Q$.}
\label{fig:Q_ex}
\end{minipage}
\begin{minipage}[t]{0.48\textwidth}
\centering
\begin{tikzpicture}[tdplot_main_coords,scale=3.9]

\coordinate (A) at (0,0.05,0.03);
\coordinate (B) at (0.7,0,0);
\coordinate (C) at (0.4,0.8,0);


\coordinate (F) at (0.8,0.4,0.65);
\coordinate (G) at (0.2,0.17,0.5);
\coordinate (H) at (0.55,0.17,0.45);

\coordinate (D) at (0.35,0.33,-0.60);

\fill[gray!20] (F)--(G)--(H)--cycle;
\fill[gray!20] (A)--(B)--(H)--(G)--cycle;
\fill[gray!20] (B)--(C)--(F)--(H)--cycle;
\fill[gray!10] (A)--(B)--(D)--cycle;
\fill[gray!20] (B)--(C)--(D)--cycle;
\fill[gray!20] (C)--(A)--(D)--cycle;

\draw[thick] (F)--(G)--(H)--cycle;
\draw[thick] (A)--(B)--(C);
\draw[dashed] (A)--(C);
\draw[thick] (A)--(B)--(H)--(G)--cycle;
\draw[thick] (B)--(C)--(F)--(H)--cycle;
\draw[thick] (A)--(D)--(B);
\draw[thick] (B)--(D)--(C);
\draw[thick] (C)--(D)--(A);

\node[font=\large] at (0.4,0.5,0.16) {1};
\node[font=\large] at (0.3,0.12,0.2) {2};
\node[font=\large] at (0.7,0.12,0.6) {3};
\node[font=\large] at (0.9,0.6,0.25) {4};
\node[font=\large] at (0.3,0.85,-0.37) {5};
\node[font=\large] at (0.15,0.55,-0.36) {6};
\node[font=\large] at (0.4,0.11,-0.2) {7};
\draw[->, dashed]  (0.9,0.34,0.47) to[out=0, in=100]  (0.9,0.54,0.3);
\draw[->, dashed]  (0.15,0.7,-0.2) to[out=0, in=100]  (0.15,0.9,-0.36);
\end{tikzpicture}
\caption{Shelling of $Q$.}
\label{fig:Q_shelling}
\end{minipage}
\end{figure}

This polytope $Q$ can be viewed as taking the bipyramid over a triangle and truncating one of the two newly added vertices. In general, any \il{S}-shelling will induce an \il{S}-partition. This implies that for polytopes, one can obtain an \il{S}-partition by taking the partition induced by a shelling. In Figures~\ref{fig:Q_shelling} and ~\ref{fig:partition_Q}, we include diagrams of a shelling of $Q$ and its corresponding partition. Notice that each partition class will contain a unique facet of $Q$. 

\tdplotsetmaincoords{70}{61}
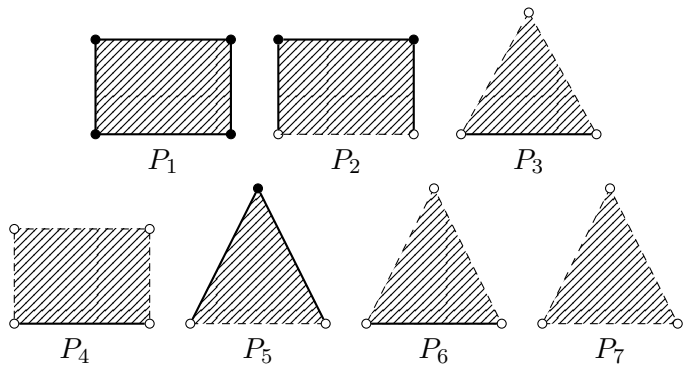
\begin{figure}[htbp]
\centering
\begin{tikzpicture}[scale=1.8]
\tikzset{
  open/.style={circle,draw,fill=white,inner sep=1.2pt},
}
\tikzset{
  open/.style={circle,draw,fill=white,inner sep=1.2pt},
  filled/.style={circle,draw,fill=black,inner sep=1.2pt},
}

\begin{scope}[shift={(-0.5,0)}]
\draw[pattern=north east lines, thick] (1.1,0) rectangle (2.1,0.7);
\node[filled] at (1.1,0) {};
\node[filled] at (2.1,0) {};
\node[filled] at (1.1,0.7) {};
\node[filled] at (2.1,0.7) {};
\node at (1.6,-0.2) {$P_1$};
\end{scope}

\begin{scope}[shift={(0.85,0)}]

\node[open] (2) at (1.1,0) {};
\node[open] (4) at (2.1,0) {};
\node[filled] (1) at (1.1,0.7) {};
\node[filled] (3) at (2.1,0.7) {};
\node at (1.6,-0.2) {$P_2$};

\fill[pattern=north east lines] (1.1,0)--(2.1,0)--(2.1,0.7)--(1.1,0.7)--cycle;

\draw[thick] (1) -- (2);
\draw[dashed] (2) -- (4);
\draw[thick] (4) -- (3);
\draw[thick] (3) -- (1);

\end{scope}

\begin{scope}[shift={(2.2,0)}]
\fill[pattern=north east lines] (1.1,0) -- (2.1,0) -- (1.6,0.9) -- cycle;
\node[open] (A) at (1.1,0) {};
\node[open] (B)at (2.1,0) {};
\node[open] (C) at (1.6,0.9) {};
\node at (1.6,-0.2) {$P_3$};

\draw[thick] (A) -- (B);   
\draw[dashed] (A) -- (C);           
\draw[dashed] (B) -- (C);           
\end{scope}

\begin{scope}[shift={(0,-1.4)}]
\fill[dashed,pattern=north east lines] (0,0) rectangle (1,0.7);
\node[open] (A) at (0,0) {};
\node[open] (B) at (1,0) {};
\node[open] (C) at (1,0.7) {};
\node[open] (D) at (0,0.7) {};
\node at (0.45,-0.2) {$P_4$};

\draw[thick] (A) -- (B);   
\draw[dashed] (A) -- (D);           
\draw[dashed] (B) -- (C);
\draw[dashed] (C) -- (D);
\end{scope}

\begin{scope}[shift={(1.3,-1.4)}]
\draw[dashed,pattern=north east lines] (0,0) -- (1,0) -- (0.5,1) -- cycle;
\node[open] (A) at (0,0) {};
\node[open] (B)at (1,0) {};
\node[filled] (C) at (0.5,1) {};
\draw[thick] (C) -- (B);
\draw[thick] (A) -- (C);
\node at (0.5,-0.2) {$P_5$};
\end{scope}

\begin{scope}[shift={(2.6,-1.4)}]
\draw[dashed,pattern=north east lines] (0,0) -- (1,0) -- (0.5,1) -- cycle;
\node[open] (A) at (0,0) {};
\node[open] (B) at (1,0) {};
\node[open] at (0.5,1) {};
\draw[thick] (A) -- (B);
\node at (0.5,-0.2) {$P_6$};
\end{scope}

\begin{scope}[shift={(3.9,-1.4)}]
\draw[dashed,pattern=north east lines] (0,0) -- (1,0) -- (0.5,1) -- cycle;
\node[open] at (0,0) {};
\node[open] at (1,0) {};
\node[open] at (0.5,1) {};
\node at (0.5,-0.2) {$P_7$};
\end{scope}
\end{tikzpicture}
\caption{\il{S}-partition of $Q$.}
\label{fig:partition_Q}
\end{figure}

To compute the \g{cd}-index of $Q$, we will use the \g{cd}-indices of each complex $\Gamma_{\sigma_i}\coloneq \overline{P_i \setminus \{\sigma_i\}}$ for each partition class $P_i$ and each associated facet $\sigma_i$ as in Figure~\ref{fig: Contribution-classes}. By summing the contributions of each facet, we find that the \g{cd}-index of $Q$ is 
$$\g{c}^3+5\g{cd}+5\g{dc}.$$

Each facet of $Q$ contributes a set of elements of $\Delta(Q)$ in such a way that each facet will contribute a non-negative sum of \g{cd}-words to the \g{cd}-index of $Q$. In general, within the semisuspension $\tilde{\Sigma}\Gamma_{\sigma_i}$, the contribution of the newly added facet $\tau_i$ will be exactly all chains in its closure as well as one $(K \cup \{d\})$-chain for each $K$-chain of $\partial \tau_i$. The \g{cd}-index of any \il{S}-partitionable poset is computed as follows:
\begin{enumerate}
    \item The contribution of $\sigma_1$ will be $\Phi(\partial\overline{\sigma_1})\cdot \g{c}$.
    \item For $2\leq i \leq t-1$ the contribution of $\sigma_i$ will be 
    \begin{enumerate}
        \item $\Phi(\partial \Gamma_{\sigma_i})\cdot \g{d}$.
        \item the contribution(s) of all facets not equal to $\tau_i$ in $\tilde{\Sigma}\Gamma_{\sigma_i}$ times $\g{c}$. 
    \end{enumerate}
    \item The contribution of $\sigma_t$ will be empty. 
\end{enumerate}

\begin{figure}[htbp]
\centering

\begin{tikzpicture}[scale=0.71]

\tikzset{
  open/.style={circle,draw,fill=white,inner sep=1.2pt},
  filled/.style={circle,draw,fill=black,inner sep=1.2pt},
  filledO/.style={circle,draw,fill=orange,inner sep=1.2pt},
  filledY/.style={circle,draw,fill=yellow,inner sep=1.2pt},
  filledP/.style={circle,draw,fill=purple,inner sep=1.2pt},
  filledC/.style={circle,draw,fill=cyan,inner sep=1.2pt}
}

\begin{scope}[shift={(-6,0)}]
\node at (-3.3,1.7) {$\Gamma_{\sigma_1} :$};
\node[filled] (A) at (-1.7,0) {};
\node[filled] (B) at (0.3,0) {};
\node[filled] (S) at (0.3,2) {};
\node[filled] (D) at (-1.7,2) {};

\draw (A) -- (B); 
\draw (B) -- (S);
\draw (S) -- (D);
\draw (A) -- (D);

\node at (1.3,1) {$\Rightarrow$};

\node at (4.75,1) {$\Phi(\square)\cdot \g{c} = \g{c}^3 + 2\g{dc}$,};
\end{scope}

\begin{scope}[scale=1, shift={(-6,0)}]
\node at (-3.3,-1.2) {$\textcolor{purple}{\tilde{\Sigma}}\Gamma_{\sigma_2}:$};
\draw[thick,purple] (-1.7,-2.7) to[out=330,in=210] (0.3,-2.7);
\node[filled] (A) at (-1.7,-2.7) {};
\node[filled] (B) at (-1.7,-1) {};
\node[filled] (T) at (0.3,-1) {};
\node[filled] (D) at (0.3,-2.7) {};
\draw (A) -- (B); 
\draw (B) -- (T);
\draw (T) -- (D);
\node at (-0.7,-3.3) {\textcolor{purple}{$\tau_2$}};
\node at (1.3,-2) {$\Rightarrow$}; 
\draw[thick,purple] (2.2,-2.7) to[out=330,in=210] (4.2,-2.7);
\node[filledP] (E) at (2.2,-2.7) {};
\node[filledP] (F) at (4.2,-2.7) {};
\node[filledC] (G) at (2.2,-1) {};
\node[filledC] (H) at (4.2,-1) {};
\node[filledP] (I) at (2.2,-2) {};
\node[filledC] (J) at (4.2,-2) {};
\node[filledC] (K)at (3.2,-1) {};
\node[filledP] (L) at (3.2,-3) {};
\draw[purple] (E) -- (I);
\draw[cyan] (I) -- (G);
\draw[purple] (F) -- (J);
\draw[cyan] (J) -- (H);
\draw[cyan] (G) -- (K);
\draw[cyan] (K) -- (H);
\node[cyan,right] at (4.4,-1.8) {Contribution of rest};
\node[purple,right] at (4.4, -3.3) {Contribution of $\tau_2$};

\node at (12.9,-1.5) {$\Phi(\partial\Gamma_{\sigma_2})\cdot \g{d} = \g{cd}$,};
\node at (13.3,-2.5) {$\Phi(\textcolor{cyan}{\square})\cdot \g{c}=2\g{dc}$,};
   
\end{scope}

\begin{scope}[shift={(-6,0.229)}]
 
\node at (-1.5,-4.22) {$\textcolor{purple}{\tilde{\Sigma}}\Gamma_{\sigma_3} =\textcolor{purple}{\tilde{\Sigma}}\Gamma_{\sigma_4}=\textcolor{purple}{\tilde{\Sigma}}\Gamma_{\sigma_6}:$};

  \node at (-0.7,-6.4) {\textcolor{purple}{$\tau_i$}};

  \draw (-1.7,-5)--(0.3,-5); 
  \draw[purple] (-1.7,-5) arc[start angle=-180,end angle=0,radius=1] -- (0.3,-5);
  \node[filled] (A) at (-1.7,-5) {};
  \node[filled] (B) at (0.3,-5) {};
  \node at (1.3,-5.1) {$\Rightarrow$};
\end{scope}

\begin{scope}[shift={(-6,0.229)}]

  \draw[purple] (2.2,-5)-- (4.2,-5); 
  \draw[purple] (2.2,-5) arc[start angle=-180,end angle=0,radius=1] -- (4.2,-5);
   \node[filledP] (A) at (2.2,-5) {};
  \node[filledP] (B) at (4.2,-5) {};
  \node[filledP] (C) at (3.2,-6) {}; 
  \node[filledP]  at (3.2,-5) {};
  \node[purple,right] at (4.4,-5.7) {Contribution of $\tau_i$};
  \node at (12.9,-5) {$\Phi(\partial\Gamma_{\sigma_i})\cdot \g{d} = \g{cd}$,};
\end{scope}


\begin{scope}[shift={(-6,1.3)}]
\node at (-3.3,-8.2) {$\textcolor{purple}{\tilde{\Sigma}}\Gamma_{\sigma_5}:$};
\draw[thick,purple] (-1.7,-10) to[out=330,in=210] (0.3,-10);
 \node[filled] (L) at (-1.7,-10) {};
 \node[filled] (M) at (0.3,-10) {};
 \node[filled] (N) at (-0.7,-8) {};
\draw (N) -- (M);
 \draw (L) -- (N);
\node at (-0.7,-10.7) {\textcolor{purple}{$\tau_5$}};
\node at (1.3,-9) {$\Rightarrow$};
\end{scope}

\begin{scope}[shift={(-5.6,1.3)}]

\draw[thick,purple] (1.8,-10) to[out=330,in=210] (3.8,-10);
 \node[filledP] (L) at (1.8,-10) {};
 \node[filledP] (M) at (3.8,-10) {};
 \node[filledC] (N) at (2.8,-8) {};
 \node[filledP] (Ñ) at (2.8,-10.3) {};
 \node[filledC] (O) at (3.3,-9) {};
 \node[filledP] (P) at (2.3,-9) {};

 \draw[purple] (M) -- (O);
 \draw[cyan] (O) -- (N);
 \draw[cyan] (N)-- (P);
 \draw[purple] (P) -- (L);

 \node[purple] at (6.6,-10.2) {Contribution of $\tau_5$};
 \node[cyan] at (6.8,-8.8) {Contribution of rest};
 \node at (12.5,-8.5) {$\Phi(\partial\Gamma_{\sigma_5})\cdot \g{d} = \g{cd}$,};
\node at (12.75,-9.5) {$\Phi(\textcolor{cyan}{\square})\cdot \g{c}=\g{dc}$,};
\end{scope}

\begin{scope}[shift={(-5.8,-3.5)}]
\node at (-3.2,-7) {$\Gamma_{\sigma_7} =\emptyset$};
\node at (1.1,-7) {$\Rightarrow$};
\node at (7.1,-7) {No contribution to the \g{cd}-index of $Q$.};
\end{scope}
\end{tikzpicture}
\caption{Contribution of the \il{S}-partition classes of $Q$.}
\label{fig: Contribution-classes}
\end{figure}
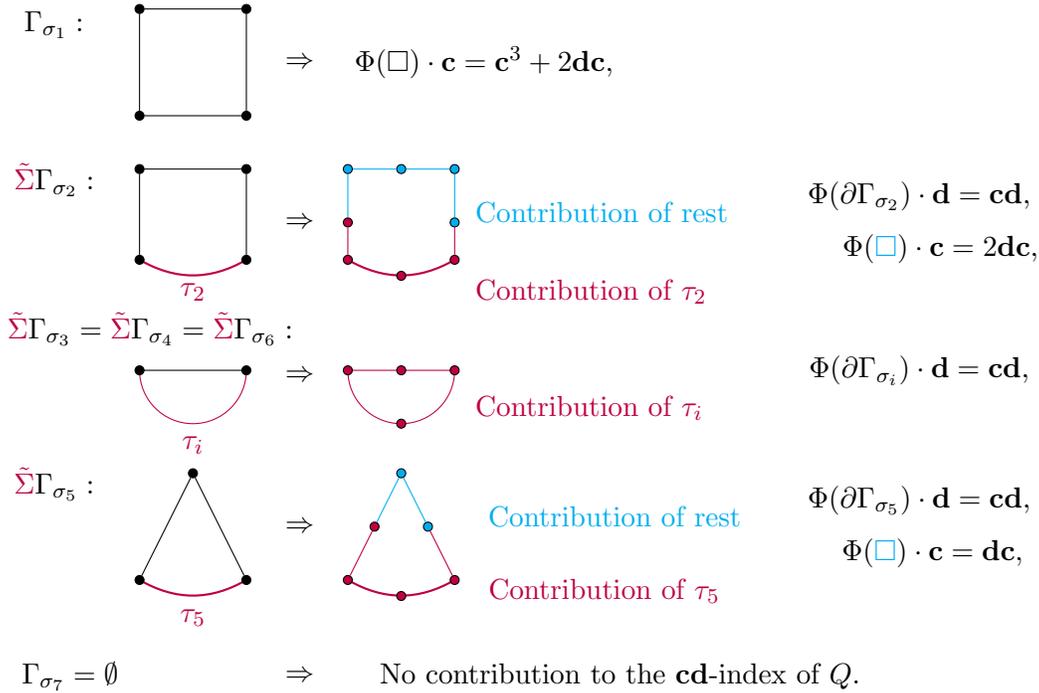

We take a moment to provide a more in-depth description of the computation of (2b) for the complex $\tilde{\Sigma}\Gamma_{\sigma_2}$. The boundary of $\tau_2$ consists of two vertices. Therefore, $\Delta(\partial \tau_2)$ consists of the empty chain and two vertices acting as singleton chains. Hence, outside of the closure of $\tau_i$, its contribution also consists of a chain consisting of a single edge as well as two vertex-edge chains. These are illustrated by the two red edges and one red vertex outside of $\Delta(\overline{\tau_i})$ in $\Delta(\tilde{\Sigma}\Gamma_{\sigma_2})$. The remaining faces are colored blue and these blue faces provide a chain polynomial of $2\g{ba}+2\g{ab}+4\g{bb}$. Converting this to the \g{ab}-polynomial, we obtain $2\g{b}(\g{a-b})+2(\g{a-b})\g{b}+4\g{bb}=2(\g{ab}+\g{ba})$. In turn, this gives us a \g{cd}-index of $2\g{d}$ from the contribution of the blue faces of the order complex of $\tilde{\Sigma}\Gamma_{\sigma_2}$.

\subsection{S-partitionable Definition}
We now  define (Eulerian) \il{S}-partitionable posets. While the definition is highly inductive, such decomposition properties are common in the literature as illustrated in Subsections~2.3 and 2.4 of Hachimori~\cite{Hachimori_2000}.

\begin{definition}\label{def:S_part}
Let $\Omega$ be an Eulerian poset of rank $d+1$. We say that $\Omega$ is \textit{\il{S}-partitionable} if either $\Omega=\{\emptyset\}$ or $\Omega$ has rank at least 2 and there exists a partition of the elements of $\Omega$ with a partition class $P_\sigma \subseteq \overline{\sigma}$ for each coatom $\sigma$ of $\Omega$. Furthermore, for $\Gamma_\sigma \coloneq \overline{P_\sigma \setminus \{\sigma\}}$ we require:
        \begin{enumerate}[(a)]
            \item There is a unique \textit{initial} coatom $\sigma_1$ for which $P_{\sigma_1}=\overline{\sigma_1}$ and we have that $P_{\sigma_1}\setminus \{\sigma_1\}=\partial \overline{\sigma_1}$ is \il{S}-partitionable of rank $d$.
            \item There is a unique \textit{terminal} coatom $\sigma_t$ for which $P_{\sigma_t}= \{\sigma_t\}$ is a one element set. 
            \item For every other coatom $\sigma$, $\Gamma_\sigma$ is near-Eulerian of rank $d$ with $\Gamma_\sigma = (P_\sigma \setminus \{\sigma\}) \uplus \partial \Gamma_\sigma$, and $\tilde{\Sigma}\Gamma_\sigma$ is \il{S}-partitionable of rank $d$ such that the initial coatom is the added coatom $\tau_\sigma$.
        \end{enumerate}
\end{definition}
We will say that the coatoms and partition classes described in part (c) of Definition~\ref{def:S_part} are \textit{ordinary}. For every ordinary partition class, there will be two associated \il{S}-partitionable posets of smaller rank.
\begin{remark}\label{rem:boundary_is_S_part}
    For every ordinary coatom $\sigma$ of an \il{S}-partition, the poset $\partial \Gamma_\sigma$ is \il{S}-partitionable of rank $d-1$ because of property $(a)$ of Definition~\ref{def:S_part} applied to $\tilde{\Sigma}\Gamma_\sigma$.
\end{remark}

\begin{remark}\label{rem:facets_part}
    Observe that because $P_\sigma \subseteq \overline{\sigma}$, we find that $\sigma \in P_\sigma$ for each coatom $\sigma$ of $\Omega$. This follows from the fact that $\sigma \in \overline{\sigma'}$ if and only if $\sigma=\sigma'$. 
\end{remark}

\subsection{S-partitions and the \g{cd}-index}

We will prove that \il{S}-partitionable posets have a non-negative \g{cd}-index by constructing a partition of the chains of $\Omega$ in a way that is inspired by Lee~\cite{Lee2010}. Our construction will rely on two inductive properties that will be applied to certain collections of chains that we call pre-blocks. Each pre-block will contain a unique chain that does not contain a coatom of $\Omega$. Our recursion will satisfy the following two inductive properties:

 \begin{enumerate}
    \item[(P1)]  For $d\geq 0$, every chain $G$ not containing a coatom of $\Omega$, will be associated to two chains $\beta(G)$ and $\tau(G)$, called the bottom and the top chain respectively, whose rank set is $r(G)\cup \{d\}$.
     \item[(P2)]  All pre-blocks associated to a coatom $\sigma$ of $\Omega$ will contribute a (possibly empty) collection of non-negative \textbf{cd}-words to the \textbf{cd}-index of $\Omega$. 
 \end{enumerate}
 
Analogous to Lee~\cite{Lee2010}, our pre-blocks will consist of either three or four chains of $\Omega$, containing $G, \beta(G)$, $\tau(G)$ and possibly an additional middle chain. We will then merge pre-blocks together in a way that gives a recursion for the \g{cd}-index in terms of \il{S}-partitionable posets of smaller rank, which will imply the non-negativity of the $\g{cd}$-index of $\Omega$. Figure~\ref{fig:preblocks} illustrates these pre-blocks for a possible part in some \il{S}-partition.

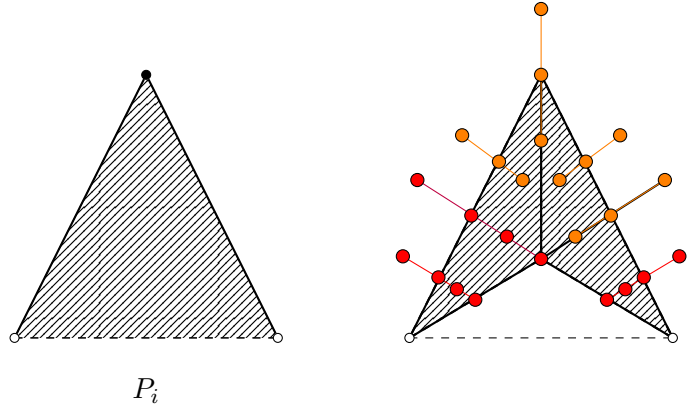
\begin{figure}[h]
    \centering
    \begin{tikzpicture}[scale=3.5]
        \tikzset{
  open/.style={circle,draw,fill=white,inner sep=1.2pt},
}
\tikzset{
  open/.style={circle,draw,fill=white,inner sep=1.2pt},
  filled/.style={circle,draw,fill=black,inner sep=1.2pt},
}
\begin{scope}   
\draw[dashed,pattern=north east lines] (0,0) -- (1,0) -- (0.5,1) -- cycle;
\draw[dashed] (0,0)--(1,0);
\node[open] (A) at (0,0) {};
\node[open] (B)at (1,0) {};
\node[filled] (C) at (0.5,1) {};
\draw[thick] (C) -- (B);
\draw[thick] (A) -- (C);
\node at (0.5,-0.2) {$P_i$};
\end{scope}

\begin{scope}[shift={(1.5,0)}]

\draw[dashed] (0,0)--(1,0);
\coordinate (A) at (0,0) {};
\coordinate (B) at (1,0) {};
\node[filled] (C) at (0.5,1) {};
\coordinate (D) at (0.5,0.3);
\coordinate (E) at (0.25,0.145);
\coordinate (F) at (0.75,0.145);
\coordinate (F') at (1.025,0.31);
\coordinate (E') at (-0.025,0.31);
\coordinate (D') at (0.03,0.6);
\coordinate (G) at (0.43,0.6);
\coordinate (H) at (0.57,0.6);
\coordinate (G') at (0.2,0.77);
\coordinate (H') at (0.8,0.77);
\coordinate (I) at (0.5,0.75);
\coordinate (I') at (0.5,1.25);
\coordinate (D'') at (0.97,0.6);
\coordinate (W) at (0.63,0.385);
\draw[dashed,pattern=north east lines,thick] (0,0) -- (D) -- (0.5,1) -- cycle;
\draw[dashed,pattern=north east lines,thick] (1,0) -- (D) -- (0.5,1) -- cycle;

\draw[thick] (D)--(D'');
\draw[thick] (C) -- (B);
\draw[thick] (A) -- (C);
\draw[thick] (A)--(D)--(B);
\draw[thick] (D)--(C);

\draw[red] (E)--(E');
\draw[red] (F)--(F');
\draw[purple] (D)--(D');
\draw[orange] (G)--(G');
\draw[orange] (H)--(H');
\draw[orange] (I)--(I');
\draw[orange] (W)--(D'');

\node at (0.5,-0.2) {};
\filldraw[fill=red] (D) circle (0.025);
\filldraw[fill=red] (D') circle (0.025);
\filldraw[fill=red] (E) circle (0.025);
\filldraw[fill=red] (F) circle (0.025);
\filldraw[fill=red] (F') circle (0.025);
\filldraw[fill=red] (E') circle (0.025);
\filldraw[fill=orange] (G) circle (0.025);
\filldraw[fill=orange] (H) circle (0.025);
\filldraw[fill=orange] (G') circle (0.025);
\filldraw[fill=orange] (H') circle (0.025);
\filldraw[fill=orange] (I) circle (0.025);
\filldraw[fill=orange] (I') circle (0.025);
\filldraw[fill=orange] (D'') circle (0.025);
\filldraw[fill=orange] (W) circle (0.025);
\filldraw[fill=white] (A) circle (0.0175);
\filldraw[fill=white] (B) circle (0.0175);

\filldraw[fill=red] (0.11,0.23) circle (0.025);
\filldraw[fill=red] (0.89,0.23) circle (0.025);
\filldraw[fill=red] (0.82,0.185) circle (0.025);
\filldraw[fill=red] (0.18,0.185) circle (0.025);
\filldraw[fill=orange] (0.34,0.67) circle (0.025);
\filldraw[fill=orange] (0.67,0.67) circle (0.025);
\filldraw[fill=orange] (0.5,1) circle (0.025);

\filldraw[fill=red] (0.37,0.385) circle (0.025);
\filldraw[fill=red] (0.235,0.465) circle (0.025);
\filldraw[fill=orange] (0.765,0.465) circle (0.025);

\end{scope}
\end{tikzpicture}
    \caption{A partition class and its pre-blocks,}
    \label{fig:preblocks}
\end{figure}
Observe that each face of the barycentric subdivision represents a chain of $\Omega$. Therefore, this illustration highlights our pre-blocks of size $3$ or $4$ and their associated chains. Notice that the bottommost element of each pre-block of size $4$ in Figure~\ref{fig:preblocks} represents a chain whose deletion of the coatom $\sigma$ is a chain in $\partial \Gamma_\sigma$. This will be how we acquire $\Phi(\partial\Gamma_\sigma)\cdot \g{d}$ in the computation. Observe that each pre-block contains a unique chain $G$ that does not contain the coatom $\sigma$. Furthermore, notice how there is a single chain that lies outside of the coatom $\sigma$ for each pre-block; this will represent the top face of $G$ and will be acquired from our placeholders and the generalized Dehn--Sommerville equations. 

Although this construction is similar to that of Lee~\cite{Lee2010}, there are substantial differences. First, the chain $\tau(G)$ is not explicit in our construction. Instead, we take advantage of the fact that to determine the flag $f$-vector of an Eulerian poset of rank $d+1$, it is enough to compute its values for all subsets of $[d-1]$ by virtue of the generalized Dehn--Sommerville equations. Thus, we do not explicitly assign $\tau(G)$, but rely on established numerical information to complete the partition. The key insight is that our pre-blocks are in bijection with the chains of $\Omega$ whose rank set does not contain $d$, so our partition contains enough information to determine the \g{cd}-index. In some cases, which include all \il{S}-shellable spheres, we can make $\tau(G)$ explicit, as discussed in Remark~\ref{rem:rev_part}. Additionally, our construction is much more general than Lee~\cite{Lee2010}. We do not require convexity, shellability, dualization, nor that $\Omega$ is a lattice.

Before embarking on the proof, we summarize it in a sequence of five steps. 
\begin{enumerate}
    \item[] \textbf{Step 1:} Verify the claims for the base cases of $\Omega=\{\emptyset\}$ and for $\Omega$ equal to a diamond poset with two elements of rank one.
    \item[] \textbf{Step 2:} Form pre-blocks of size $3$ given by $\{\beta(G),G,r(G)\cup \{d\}\}$ consisting of $G$, a chain of $\Omega$ with $d\not \in r(G)$, $\beta(G)$, an extension of $G$ induced by the \il{S}-partition, and a placeholder for an $(r(G)\cup \{d\})$-chain of $\Omega$.
    \item[] \textbf{Step 3:} Turn some of the existing pre-blocks of size $3$ into pre-blocks of size $4$ by adding a new chain of $\Omega$ to an existing pre-block. Property (P1) will be used here to ensure that this assignment is well-defined and does not create pre-blocks of size larger than $4$.
    \item[] \textbf{Step 4:} Argue why the generalized Dehn--Sommerville equations guarantee that the existence of any \g{cd}-index of our pre-blocks must match the \g{cd}-index of $\Omega$.
    \item[] \textbf{Step 5:} Recursively compute a non-negative \g{cd}-index of our pre-blocks. Property (P2) will be used here to ensure non-negativity.
    \end{enumerate}
With these steps in mind, we begin the proof of the non-negativity of the \g{cd}-index for all \il{S}-partitionable posets. 
\begin{theorem}\label{Th.S-Partitionable}
           The \g{cd}-index of an \il{S}-partitionable poset can be recursively computed in terms of its \il{S}-partitionable parts of smaller rank and is consequently non-negative. 
\end{theorem}

\begin{proof}
    Let $\Omega$ be an \il{S}-partitionable poset of rank $d+1$. If $d=0$, then $\Delta(\Omega)$ is decomposed into a single block $\emptyset$ that has a \textbf{cd}-index equal to 1. If $d=1$, then $\Omega$ is a diamond poset with two elements of rank $1$ given by $\sigma_1,\sigma_2$, because it is Eulerian. It has two partition classes $[\hat 0, \sigma_1]$ and $\{\sigma_2\}$. In the order complex, there are three chains $\emptyset$, $\{\sigma_1\}$ and $\{\sigma_2\}$. Here, $\{\sigma_1\}$ is the bottom chain of the empty chain, while $\{\sigma_2\}$ is the top chain of the empty chain. This yields a single block $\{\sigma_1, \emptyset, \sigma_2\}$. These three chains together contribute the \textbf{cd}-word \g{c} and we say that $\sigma_1$ contributes this word to the \textbf{cd}-index of $\Omega$. 
    
    Assume $d\geq 2$ and that properties (P1) and (P2) hold for all \il{S}-partitionable posets of rank at most $d$.
    \medskip
    \newline
    \textbf{Size $3$ pre-blocks:} Let $G$ be an element of $\Delta(\Omega)$ such that $d\not \in r(G)$. Consider the element of $G$ with the largest rank in $\Omega$. This element belongs to a unique partition class $P_\sigma$ of our \il{S}-partition. Define $\beta(G)$, the bottom chain of $G$, to be the chain (of length one longer) obtained by appending $\sigma$ to $G$. Observe that $\beta(G)=\beta(G')$ implies $G=G'$. The pre-block of size $3$ associated to $G$ will contain  $\beta(G),G$ and a placeholder for an $(r(G)\cup \{d\})$-chain of $\Omega$. One should think of this placeholder as a ``token" for an $(r(G) \cup \{d\})$-chain of $\Omega$ that will be redeemed at a later step. This concludes our formation of pre-blocks of size $3$ that are all of the form:
   $$\{\beta(G),G,r(G)\cup \{d\}\}_{G:d\not \in r(G)}.$$
   \medskip
   \newline
   \textbf{Size $4$ pre-blocks:} Let $F$ be an element of $\Delta(\Omega)$ that satisfies $d\in r(F)$, and upon removing the unique coatom $\sigma$ contained in $F$ to create the chain $\hat{F}$, we have that $\hat{F}\in \Delta(\partial\Gamma_\sigma)$. The chains of this form are called \textit{middle chains of $\sigma$}. Notice that since $\partial \Gamma_\sigma$ has rank $d-1$, it must be that $d-1\not \in r(\hat{F})$. By the definition of \il{S}-partitionability, the poset, $\tilde{\Sigma}\Gamma_\sigma$, is \il{S}-partitionable of rank $d$ with the newly added coatom $\tau_\sigma$ corresponding to the initial partition class. Furthermore $\partial \tau_\sigma=\partial \Gamma_\sigma$. By property (P1), we find that $\hat{F}$ has an associated top chain $\tilde{F}$ in $\tilde{\Sigma}\Gamma_\sigma$ with rank set $r(\hat{F})\cup \{d-1\}$. Notice that this top chain cannot contain $\tau_\sigma$, because all chains containing $\tau_\sigma$ are bottom chains of some chain in $\partial \Gamma_\sigma$. Thus, by definition of $\tilde{\Sigma}\Gamma_\sigma$, the top chain of $\hat{F}$ in $\tilde{\Sigma}\Gamma_\sigma$ contains an element of rank $d-1$ of $\Omega$ belonging to $P_\sigma$ as these are the only remaining elements of rank $d-1$ of $\tilde{\Sigma}\Gamma_\sigma$. Insert $F$ into the pre-block of size $3$ of $\Omega$ containing $\tilde{F}$ to create a pre-block of size $4$ of the form:
   $$\{F,\beta(G),G=\tilde{F},r(G)\cup \{d\}\}.$$
   
   Next, we show that these pre-blocks truly are of size four in that no two middle chains are placed into the same pre-block of size $3$. Assume, for the sake of contradiction, that $F_1$ and $F_2$ are placed into the same pre-block of size $3$. Then $\tilde{F}_1=\tilde{F}_2$ contains an element of $\Omega$ of rank $d-1$ belonging to $P_\sigma$ for some ordinary coatom $\sigma$. Thus, by deleting the rank $d-1$ element from these two chains, we find that $\hat{F}_1=\hat{F}_2$. Furthermore, both $F_1$ and $F_2$ must contain the coatom $\sigma$ of $\Omega$. Indeed, if $F$ is any middle chain of $\sigma$, then in the poset $\tilde{\Sigma}\Gamma_\sigma$, the chain $\tilde{F}$ is a chain that contains an element of $\Omega$ of rank $d-1$ belonging to $P_\sigma$ (and only $P_\sigma$). Hence, by appending $\sigma$ to $\hat{F}_1=\hat{F}_2$, we recover $F_1=F_2$ as desired. 
  \medskip
  \newline
   \textbf{Placeholders:}
   At this point, we have pre-blocks of size $3$ and $4$ of the form:
$$\{\beta(G),G,r(G)\cup \{d\}\}$$
$$\{F,\beta(G),G,r(G) \cup \{d\}\}.$$
We will redeem the placeholders in these pre-blocks with actual chains of $\Omega$ with the appropriate rank set. Because $\Gamma_\sigma = (P_\sigma \setminus \{\sigma\}) \uplus \partial \Gamma_\sigma$, we find that no middle chains can also be bottom chains. Therefore, all elements of $\Delta(\Omega)$ are currently used in at most one pre-block. Therefore, to make a rank-set preserving bijection between our pre-blocks and the order complex of $\Omega$, it suffices to verify that the number of $K$-chains of $\Omega$ equals the number of $K$-chains in our pre-blocks (with placeholders) for each $K\subseteq [d]$. 

By construction, we know that the number of $K$-chains of $\Omega$ equals the number of $K$-chains in our pre-blocks for $K \subseteq [d-1]$.
Hence, to verify that our pre-blocks (with placeholders) contain the same number of chains as $\Omega$ for each rank set, it suffices to prove that our collection of pre-blocks satisfies the generalized Dehn--Sommerville equations. Two vectors that satisfy these equations and have the same entries for $K\subseteq [d-1]$ must be equal. 
It follows from Theorem $\ref{thm: BayerKlapper}$ that computing a \textbf{cd}-index of these pre-blocks implies that there is a rank-set preserving bijection between the placeholders in our pre-blocks and the unassigned elements of $\Delta(\Omega)$. Choosing an arbitrary rank-set preserving bijection will establish $\tau(G)$ for each $G$, and complete the desired partition into our desired blocks of the faces of the order complex of $\Omega$.
\medskip
\newline
\textbf{The \textbf{cd}-index:} Our goal is to show that the constructed pre-blocks yield a non-negative \textbf{cd}-index. Each pre-block has a unique coatom $\sigma$ of $\Omega$ associated to it given by the element of rank $d$ contained in the corresponding $\beta(G)$. To this end, we say that the contribution of $\sigma$ to the \textbf{cd}-index of $\Omega$ will be the set of \textbf{cd}-words generated by the pre-blocks associated to $\sigma$. Notice that $\sigma_t$ will have no associated pre-blocks; it will have an empty contribution to the \g{cd}-index of $\Omega$. Naturally, all pre-blocks will come in two forms, those of size $3$ and those of size $4$. It remains to show that each non-terminal coatom $\sigma$ of $\Omega$ contributes a (nonempty) sum of \textbf{cd}-words with positive integral coefficients. In doing so, by summing together all of these \textbf{cd}-words, we will acquire the \textbf{cd}-index of $\Omega$ as desired. 

Fix an ordinary coatom $\sigma$ of $\Omega$. We begin by counting the contributions of $\sigma$ from its pre-blocks of size four. Consider all pre-blocks of the form $\{F,\beta(G),G,r(G) \cup\{d\}\}$ with $\sigma \in \beta(G)$. Let $S_f(\hat{F})$ (resp. $S_h(\hat{F}))$ be the flag $f$- (resp. flag $h$-) word of $\hat{F}$ in noncommuting variables \g{a} and \g{b} in the poset $\partial\Gamma_\sigma$ of rank $d-1$. Then, the pre-block of size four given by $\{F,\beta(G),G,r(G)\cup \{d\}\}$ contributes the following to the chain polynomial of $\Omega$. 
$$S_f(\hat{F})\g{ab}+S_f(\hat{F})\g{bb}+S_f(\hat{F})\g{ba}+S_f(\hat{F})\g{bb}$$
By applying the transformation $\g{a}\mapsto \g{a}-\g{b}$ and $\g{b} \mapsto \g{b}$, we recover the contribution to the \g{ab}-polynomial of $\Omega$ as follows:
$$S_h(\hat{F})(\g{a}-\g{b})\g{b}+S_h(\hat{F})\g{bb}+S_h(\hat{F})\g{b}(\g{a}-\g{b})+S_h(\hat{F})\g{bb}=S_h(\hat{F})(\g{ab}+\g{ba}).$$

Observe that $\partial \Gamma_\sigma$ is \il{S}-partitionable of rank $d-1$ by Remark~\ref{rem:boundary_is_S_part}. Therefore, by the induction hypothesis, it has a non-negative \textbf{cd}-index. Furthermore, because there is a bijection (given by adding/deleting the element $\sigma$) between middle chains of $\sigma$ and elements of $\Delta(\partial \Gamma_\sigma)$, we conclude that all of the pre-blocks of size 4 associated to $\sigma$ yield $\Phi(\partial \Gamma_\sigma)\cdot \mathbf{d}$. Therefore, the collection of all pre-blocks of size $4$ gives a contribution of 
$$\sum_{\substack{\sigma \text{ ordinary}\\ \text{in $\Omega$} }}\Phi(\partial \Gamma_\sigma)\cdot \mathbf{d}$$
to the \textbf{cd}-index of $\Omega$. We reiterate that by the induction hypothesis and Remark~\ref{rem:boundary_is_S_part}, each summand is a non-negative sum of \textbf{cd}-words. Therefore, the set of all pre-blocks of size four contribute a non-negative collection of \textbf{cd}-words to the \textbf{cd}-index of $\Omega$. 

It remains to count the contributions of the pre-blocks of size 3. Consider $G$ with $d\not \in r(G)$ that is in a pre-block of size 3. Let $T_f(G)$ (resp. $T_h(G))$ be the flag $f$- (resp. flag $h$-) word in noncommuting variables \g{a} and \g{b} in $P_1\setminus\{\sigma_1\}$ or $\tilde{\Sigma}\Gamma_\sigma$ for an ordinary coatom $\sigma$. Then, the block of size three given by $\{\beta(G),G,r(G)\cup \{d\}\}$ contributes the following to the chain polynomial of $\Omega$:
$$T_f(G)\g{b}+T_f(G)\g{a}+T_f(G)\g{b}=2T_f(G)\g{b}+T_f(G)\g{a}.$$
By applying the transformation $\g{a}\mapsto \g{a}-\g{b}$ and $\g{b} \mapsto \g{b}$, we recover the contribution to the \g{a},\g{b}-polynomial of $\Omega$ as follows:
$$2T_h(G)\g{b}+T_h(G)(\g{a}-\g{b})=T_h(G)(\g{a}+\g{b}).$$
If we sum over all chains $G$ such that $\sigma_1\in \beta(G)$, we acquire a contribution of $$\Phi(\partial \overline{\sigma_1})\cdot \g{c}.$$ 

It remains to count the contributions of the remaining pre-blocks of size 3. We begin by investigating which chains $G$ with $d\not \in r(G)$ are contained in these remaining pre-blocks. Such chains have the following properties:
\begin{itemize}
    \item The highest ranked element of $G$ lies in a partition class $P_\sigma$ for some ordinary coatom $\sigma$. 
    \item The chain $G$ is not the top chain of some $\hat{F}\in \Delta(\partial \Gamma_\sigma)$ in the poset $\tilde{\Sigma}\Gamma_\sigma$. 
\end{itemize}
These remaining chains $G$ are exactly the contributions of the ordinary \il{S}-partition classes in the recursive \il{S}-partition of each $\tilde{\Sigma}\Gamma_\sigma$. By property (P2) of the induction hypothesis, the \g{cd}-index contribution of each partition class of $\tilde{\Sigma}\Gamma_\sigma$ is non-negative. Let $\mathcal{C}_{\Omega}(\sigma)$ denote the contribution of $\sigma$ to the \textbf{cd}-index of $\Omega$. We get the following recursive formula for the \g{cd}-index of $\Omega$. 
$$\Phi(\Omega)=\Phi(\partial \overline{\sigma_1})\cdot \g{c} + \sum_{\substack{\sigma \text{ ordinary}\\ \text{in $\Omega$} }}\Phi(\partial \Gamma_\sigma)\cdot \mathbf{d} + \sum_{\substack{\sigma \text{ ordinary}\\ \text{in $\Omega$} }} \ \sum_{\substack{\omega \text{ ordinary}\\ \text{in $\tilde{\Sigma}\Gamma_\sigma$} }}\mathcal{C}_{\tilde{\Sigma}\Gamma_\sigma}(\omega) \cdot \g{c}.$$
\end{proof}

Thus, we have established the non-negativity of the \g{cd}-index of \il{S}-partitionable posets by examining the contributions from each coatom $\sigma$. We follow this by showing that the blocks corresponding to a given coatom in the recursive \il{S}-partition structure exhibit a finer structure; they can be merged so that each new partition class contributes a single \g{cd}-word to the \g{cd}-index of $\Omega$. 

\begin{corollary}\label{cor:single_word}
For an \il{S}-partitionable poset $\Omega$, the blocks of $\Delta(\Omega)$ corresponding to each coatom can be merged in such way so that each new partition class contributes a single \g{cd}-word to the \g{cd}-index of $\Omega$. 
\end{corollary}

\begin{proof}
The proof is by induction on rank. For $d=1$, no refinement is necessary as $\sigma_1$ already contributes the single $\g{cd}$-word \g{c}. Assume $d\geq 1$ and that the statement is true for all \il{S}-partitionable posets of rank at most $d$. For $\sigma_1$ being the first coatom of $\Omega$, simply apply induction to the \il{S}-partitionable poset $\partial\overline{\sigma}_1$ to refine its partition into classes. Then, apply this refinement to the blocks of size three
$$\{\beta(G),G, \tau(G)\}$$
with $G\in \Delta(\partial\overline{\sigma_1})$. In doing so, each single \g{cd}-word will turn into one with \g{c} multiplied on the right. An analogous argument establishes this for all other blocks of size $3$. As for the blocks of size four, apply the refined partition of $\partial \Gamma_\sigma$ to the corresponding blocks of size $4$ given by 
$$\{F,\beta(G),G,\tau(G)\}$$
via the correspondence of adding/deleting $\sigma$ between $F$ and $\hat{F}$. Each single \g{cd}-word of $\partial \Gamma_\sigma$ will be multiplied on the right by \g{d}. Therefore, the partition of the blocks can be refined so that each partition class of the refined partition contributes a single \g{cd}-word to the \g{cd}-index of $\Omega$.
\end{proof}

\begin{remark}\label{rem:rev_part}
    Suppose that $\Omega$ is an \il{S}-partitionable poset. Define the \il{reverse partition class of $P_\sigma$} as $P_\sigma'\coloneq (\partial \sigma \setminus \Gamma_\sigma) \cup \{\sigma\}$. If the set $\{P_\sigma'\}$ also partitions the poset $\Omega$, then we may explicitly define $\tau(G)$ for each $G$ with $d \not \in r(G)$. Observe that there will be a unique partition class $P_{\sigma'}'$ in which the face of highest rank of $G$ will belong. Append $\sigma'$ to $G$ to form $\tau(G)$. Notice that because $P_\sigma \cap P_\sigma'=\{\sigma\}$ for every coatom $\sigma$, we find that $\beta(G)$ and $\tau(G)$ will contain distinct coatoms of $\Omega$. 
\end{remark}

Observe that such a reverse partitioning is always possible for a shellable polytopal complex homeomorphic to a sphere. This is because any shelling order of a sphere has the property that this total order in reverse is also a shelling order. Thus, such a reversible shelling recovers this reversible partition property. If the reverse partition classes of an \il{S}-partitionable complex $\Omega$ also form a partition of $\Omega$, then we say that $\Omega$ is \textit{reversibly partitionable}. This begs the following question:

\begin{question}
    Which \il{S}-partitionable posets are reversibly partitionable?
\end{question}

Even when restricting to the setting of Eulerian simplicial complexes, this question remains evasive. Recall that a simplicial complex is \textit{partitionable} if its face lattice can be partitioned into Boolean intervals $[R_i,\sigma_i]$ where $R_i$ is a face contained in the facet $\sigma_i$. If a simplicial complex is partitionable, we will say that each $R_i$ is a restriction face of the partition.  

\begin{question}
Given an Eulerian partitionable simplicial complex $\Omega$ with restriction faces $\{R_i\}_{i\in [t]}$ and facets $\{\sigma_i\}_{i\in [t]}$, when does the set $\{\sigma_i \setminus R_i\}_{i\in [t]}$ form a set of restriction faces of $\Omega$?
\end{question}

\subsection{Consequences}
Stanley~\cite{Stanley1994} introduced \il{S}-shellable spheres as a subclass of Eulerian posets that has a non-negative \g{cd}-index. Just as partitionable simplicial complexes generalize shellable ones, our next result shows that \il{S}-partitionable posets generalize the notion of \il{S}-shellability.
\begin{definition}[\cite{Stanley1994}] Let $\Omega$ be an Eulerian regular CW-complex of dimension $d$. We say that $\Omega$ is \textit{\il{S}-shellable} (short for ``spherically shellable'') if either $\Omega=\{\emptyset\}$, or else we can linearly order the facets (open $d$-cells) of $\Omega$, say $\sigma_1,\hdots, \sigma_t$ such that for all $1\le i \le t$ the following conditions hold (where both \ $\bar{ }$ \ and \text{cl} denote closure)
\begin{enumerate}[a)]
    \item $\partial\overline{\sigma_1}$ is \il{S}-shellable (of dimension $d-1$)
    \item For $2\le i \le t-1$, let 
    $$\Gamma_{\sigma_i}=\text{cl}[\partial \overline{\sigma_i}-((\overline{\sigma_1}\cup \hdots\cup \overline{\sigma_{i-1}})\cap \overline{\sigma_i})].$$
\end{enumerate}
Thus $\Gamma_{\sigma_i}$ is the subcomplex of $\partial \overline{\sigma_i}$ generated by all $(d-1)$-cells of $\partial\overline{\sigma_i}$ which are not contained in the complex $\overline{\sigma_1}\cup \hdots\cup \overline{\sigma_{i-1}}$ generated by previous cells. Then we require that $\Gamma_{\sigma_i}$ is near-Eulerian of dimension $d-1$, and that the semisuspension  $\tilde\Sigma\Gamma_{\sigma_i}$ is \il{S}-shellable, with first facet of the shelling being the facet $\tau=\tau_i$ adjoined to $\Gamma_{\sigma_i}$ to obtain $\tilde\Sigma\Gamma_{\sigma_i}$.
\end{definition}

As noted in Stanley~\cite{Stanley1994}, all $S$-shellable regular CW complexes are spheres. Therefore, we do not lose any generality when referring to Stanley's \il{S}-shellable complexes as \il{S}-shellable spheres. The following corollary recovers Theorem~2.2 of Stanley~\cite{Stanley1994}.

\begin{corollary} All
    \il{S}-shellable spheres are \il{S}-partitionable. In particular, \il{S}-shellable spheres have a non-negative \g{cd}-index.
\end{corollary}

\begin{proof}
    Let $\Omega$ be an \il{S}-shellable sphere. If $\Omega$ is of dimension $-1$, then $\Omega=\{\emptyset\}$ which is both \il{S}-partitionable and \il{S}-shellable. Therefore, we may suppose that $\Omega$ has non-negative dimension. In this case, consider the total order of the facets of $\Omega$ provided by the \il{S}-shelling as $\sigma_1,\sigma_2,\ldots,\sigma_t$. We shall define $P_i$ to consist of all faces that belong to $\overline{\sigma_i} \setminus \left(\cup_{j<i} \overline{\sigma_j}\right)$. In other words, $P_i$ consists of all faces of $\Omega$ that have the property that $\sigma_i$ is the first facet of the total order to which it belongs.
    
    By the definition of \il{S}-shellability, $P_1\setminus \{\sigma_1\}=\partial \overline{\sigma}_1$ is \il{S}-shellable, and thus is \il{S}-partitionable by induction. Furthermore, by the definition of \il{S}-shellability, we know that for $2\leq i \leq t-1$
    $$\Gamma_{\sigma_i}\coloneq \text{cl}[\partial\overline{\sigma_i}\setminus(\overline{\sigma_1}\cup \cdots \cup \overline{\sigma_{i-1}})]$$
    exactly equals $\overline{P_i\setminus\{\sigma_i\}}$ as in the definition of \il{S}-partitionability. From the \il{S}-shelling, we know that $\Gamma_{\sigma_i}$ is near-Eulerian of dimension $d-1$ and that the semisuspension $\tilde{\Sigma}\Gamma_{\sigma_i}$ is \il{S}-shellable with the first facet of the shelling being the newly added facet $\tau_i$. Therefore, by induction, this \il{S}-shelling on the $(d-1)$-dimensional complex $\tilde{\Sigma}\Gamma_{\sigma_i}$ induces an \il{S}-partition of dimension $d-1$ with the first partition class containing the newly added facet $\tau_i$. Finally, because all \il{S}-shellable complexes are spheres, it must be that the boundary of the last facet, $\partial\overline{\sigma_t}$, is (homeomorphically) attached to the entire boundary of the ball $\bigcup_{i=1}^{t-1}\overline{\sigma_i}$. Therefore $P_t=\{\sigma_t\}$, and we conclude that $\Omega$ is \il{S}-partitionable.
\end{proof}

\begin{corollary}\label{cor:part_implies_Spart}
   All Eulerian partitionable simplicial complexes are \il{S}-partitionable. In particular, they have a non-negative \g{cd}-index.
\end{corollary}

\begin{proof} Let $\Delta$ be an Eulerian partitionable simplicial complex. If $\Delta$ has negative dimension, then $\Omega=\{\emptyset\}$ which is both partitionable and \il{S}-partitionable. Therefore, we may suppose that $\Delta$ has dimension $d\geq 0$ with facets $\sigma_1,\hdots,\sigma_t$ and restriction faces $\{R_i\}_{i\in [t]}$. Since $\Delta$ is Eulerian, we have $h_0=h_d=1$. Also, since $\Delta$ is partitionable, the $h$-vector counts the sizes of the restriction faces, and, without loss of generality, we can take $R_1=\emptyset$ and $R_t=\sigma_t$. For each $i$, define $P_i$ to consist of all faces of the interval $[R_i,\sigma_i]$.

Note that $P_1\setminus \{\sigma_1\}=\partial \overline{\sigma_1}$ is \il{S}-shellable, and hence \il{S}-partitionable. For $2\le i\le t-1$, we find that $\Gamma_{\sigma_i}=\overline{P_i\setminus \{\sigma_i\}}$ is exactly $\text{st}_{\partial \overline{\sigma_i}}(R_i)=\{\overline{\tau}:R_i\subset \tau \subsetneq \sigma_i\}\subseteq \partial \overline{\sigma_i}$. Moreover, since any permutation of the facets of a simplex yields a shelling, we can choose a shelling of $\partial \overline{\sigma_i}$ whose final segment has $\text{st}_{\partial\overline{\sigma_i}}(R_i)$ as its closure. Because $\text{st}_{\partial\overline{\sigma_i}}(R_i)$ consists of a nonempty proper subset of facets of $\partial \overline{\sigma_i}$, we may merge all facets of $\partial \overline{\sigma_i}$ not containing $R_i$ together to obtain a ball whose boundary is $\partial \Gamma_{\sigma_i}$. In doing so, we prove that $\Gamma_{\sigma_i}$ is near-Eulerian and that the semisuspension $\Sigma\tilde\Gamma_{\sigma_i}$ is \il{S}-shellable. Hence $\Sigma\tilde\Gamma_{\sigma_i}$ is also \il{S}-partitionable as desired.
\end{proof}

\begin{remark}
    The non-negativity of the \g{cd}-index of Eulerian partitionable simplicial complexes also follows from Stanley~\cite[Theorem~3.1]{Stanley1994}, because the $h$-vectors of partitionable simplicial complexes are known to be non-negative. However, our proof does not use the $h$-vector and computes the \g{cd}-index directly. 
\end{remark}

We finish this section with a quick example. Namely, we compute the \g{cd}-index of the regular CW sphere of dimension $d$ with exactly  two cells in each (non-negative) dimension.
\begin{example}
Let $\Omega$ be the face poset of the regular CW sphere of dimension $d$ with exactly $2$ cells in each (non-negative) dimension. By induction, one may quickly show that $\Omega$ is \il{S}-shellable, and thus \il{S}-partitionable. For $d=0$, we find that $\Omega$ has a \g{cd}-index of $\g{c}$. When computing the $\g{cd}$-index for higher dimensions, we find that the formula in Theorem~\ref{Th.S-Partitionable} reduces to $\Phi(\partial \overline{\sigma_1})\cdot \g{c}$ because $\Omega$ has only two facets $\sigma_1,\sigma_2$. Furthermore, $\partial \overline{\sigma_1}$ is a regular CW sphere of dimension $d-1$ with exactly $2$ cells in each (non-negative) dimension. Therefore, by induction, the \g{cd}-index of $\Omega$ is precisely $\g{c}^{d+1}$. 
\end{example}

\begin{corollary}
Let $M, N$ be two \il{S}-partitionable regular CW complexes such that the last facet of $M$ is combinatorially isomorphic to the first facet of $N$. Then the connected sum $M\# N$ is \textit{S}-partitionable. In particular, if both $M,N$ are \il{S}-partitionable cubical (resp. simplicial) complexes, then the connected sum $M\# N$ is \il{S}-partitionable.
\end{corollary}
\begin{proof}
Form the connected sum $M\# N$ by gluing the boundary of the last facet of $M$ to the boundary of the first facet of $N$. Then, the \il{S}-partition induced from all facets not equal to these two will form an \il{S}-partition of $M \# N$. 
\end{proof}

\section{Semi-Eulerian Partitionability}\label{sec:SE_part}

Building on our notion of \il{S}-partitionable posets, we introduce a broader class of (semi-Eulerian) \il{SE}-partitionable posets. This extension will allow us to provide an explicit recursive formula to compute the semi-Eulerian \g{cd}-index as given by Definition~\ref{def:semicd}. The definition of \il{SE}-partitionable posets enables us to construct many non-Eulerian simplicial and non-simplicial examples in which the \g{cd}-index is manifestly non-negative. 

\subsection{A Motivating Example}
We begin by considering the following non-simplicial decomposition of the $2$-torus $T$.

\begin{figure}[!ht]
\centering

\begin{tikzpicture}[scale=1.3, thick]

\foreach \x in {0,1.5,3}
  \foreach \y in {2,3,4}
    \filldraw[black] (\x,\y) circle (1.3pt);

\foreach \y in {1,2,3,4}
  \draw (0,\y) -- (4.5,\y);
\foreach \x in {0,1.5,3,4.5}
  \draw (\x,1) -- (\x,4);
\draw (0,2) -- (1.5,3);
\draw (1.5,3) -- (3,4);
\draw (1.5,1) -- (3,2);
\draw (3,2) -- (4.5,3);


\draw[pattern=north east lines, pattern color=gray] (0,2)--(1.5,3)--(0,3)--cycle;
\draw[pattern=north east lines, pattern color=gray] (0,2)--(1.5,3)--(1.5,2)--cycle;

\draw[pattern=north east lines, pattern color=gray] (1.5,3)--(3,4)--(3,3)--cycle;
\draw[pattern=north east lines, pattern color=gray] (1.5,3)--(3,4)--(1.5,4)--cycle;

\draw[pattern=north east lines, pattern color=gray] (1.5,1)--(3,2)--(3,1)--cycle;
\draw[pattern=north east lines, pattern color=gray] (1.5,1)--(3,2)--(1.5,2)--cycle;

\draw[pattern=north east lines, pattern color=gray] (3,2)--(4.5,3)--(4.5,2)--cycle;
\draw[pattern=north east lines, pattern color=gray] (3,2)--(4.5,3)--(3,3)--cycle;

\draw[pattern=north east lines, pattern color=gray] (0,3) rectangle (1.5,4);
\draw[pattern=north east lines, pattern color=gray] (3,3) rectangle (4.5,4);
\draw[pattern=north east lines, pattern color=gray] (1.5,2) rectangle (3,3);
\draw[pattern=north east lines, pattern color=gray] (0,1) rectangle (1.5,2);
\draw[pattern=north east lines, pattern color=gray] (3,1) rectangle (4.5,2);

\draw[->, thick] (0,2) -- (0,2.6);
\draw[->, thick] (4.5,2.55) -- (4.5,2.55);
\draw[->>, thick] (2,4) -- (2.5,4);
\draw[->>, thick] (2,1) -- (2.5,1);

\end{tikzpicture}
    \caption{A decomposition of the $2$-torus.}
    \label{fig:decomp_2torus}
\end{figure}
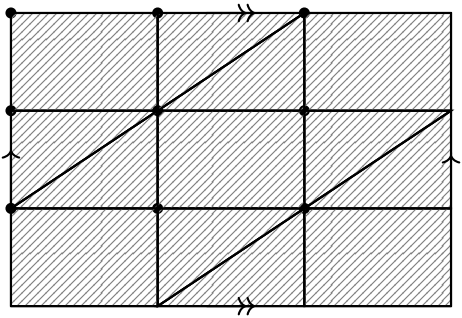

Observe that the torus has an Euler characteristic of $0$ while the sphere $\mathbb{S}^2$ has an Euler characteristic of 2. Therefore, the face poset is not Eulerian and we cannot directly produce an \il{S}-partition to compute the \g{cd}-index of $T$.  However, it is semi-Eulerian, and we can consider the partition provided in Figure~\ref{fig:SEdecomp}.

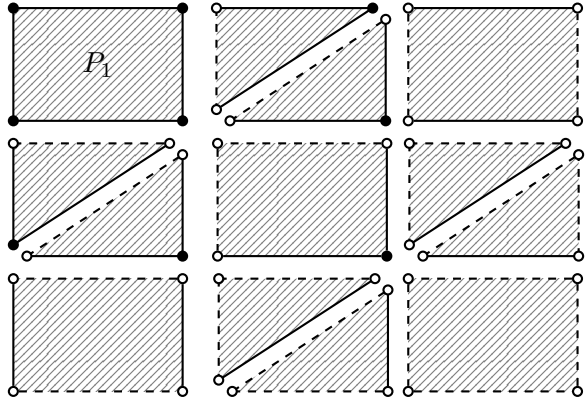
\begin{figure}[!ht]
\centering

\begin{tikzpicture}[scale=1.5, thick]
\tikzset{
  open/.style={circle,draw,fill=white,inner sep=1.2pt},
}
\tikzset{
  open/.style={circle,draw,fill=white,inner sep=1.2pt},
  filled/.style={circle,draw,fill=black,inner sep=1.2pt},
}

\draw[pattern=north east lines, pattern color=gray] (0,3) rectangle (1.5,4);
\node[filled] at (0,4) {};
\node[filled] at (0,3) {};
\node[filled] at (1.5,3) {};
\node[filled] at (1.5,4) {};
\node at (0.75,3.5) {$P_1$};

\begin{scope}

\fill[pattern=north east lines, pattern color=gray] (1.8,4)--(1.8,3.1)--(3.185,4)--cycle;

\fill[pattern=north east lines, pattern color=gray] (1.92,3)--(3.3,3)--(3.3,3.9)--cycle;

\draw[dashed] (1.8,4)--(1.8,3.1);
\draw (1.8,4)--(3.185,4);
\draw[dashed] (1.92,3)--(3.3,3.9);

\draw (1.8,3.1)--(3.185,4);
\draw (1.92,3)--(3.3,3);
\draw (3.3,3)--(3.3,3.9);

\node[open] at (1.8,4) {};
\node[open] at (1.8,3.1) {};
\node[open] at (1.92,3) {};

\node[filled] at (3.3,3) {};

\node[open] at (3.3,3.9) {};
\node[filled] at (3.185,4) {};
    
\end{scope}


\begin{scope}[shift={(3.5,3)}]
\fill[pattern=north east lines, pattern color=gray] (0,0) rectangle (1.5,1);
\node[open] (A) at (0,0) {};
\node[open] (B) at (1.5,0) {};
\node[open] (C) at (1.5,1) {};
\node[open] (D) at (0,1) {}; 

\draw (A) -- (B);
\draw[dashed] (B)--(C);
\draw (C)--(D); 
\draw[dashed] (D)--(A);
\end{scope}

\begin{scope}[shift={(-1.8,-1.2)}]

\fill[pattern=north east lines, pattern color=gray] (1.8,4)--(1.8,3.1)--(3.185,4)--cycle;

\fill[pattern=north east lines, pattern color=gray] (1.92,3)--(3.3,3)--(3.3,3.9)--cycle;

\draw (1.8,4)--(1.8,3.1);
\draw[dashed] (1.8,4)--(3.185,4);
\draw[dashed] (1.92,3)--(3.3,3.9);

\draw (1.8,3.1)--(3.185,4);
\draw (1.92,3)--(3.3,3);
\draw (3.3,3)--(3.3,3.9);

\node[open] at (1.8,4) {};
\node[filled] at (1.8,3.1) {};
\node[open] at (1.92,3) {};

\node[filled] at (3.3,3) {};

\node[open] at (3.3,3.9) {};
\node[open] at (3.185,4) {};
    
\end{scope}

\begin{scope}[shift={(1.81,1.8)}]
\fill[pattern=north east lines, pattern color =gray] (0,0) rectangle (1.5,1);
\node[open] (A) at (0,0) {};
\node[filled] (B) at (1.5,0) {};
\node[open] (C) at (1.5,1) {};
\node[open] (D) at (0,1) {}; 

\draw (A) -- (B);
\draw (B)--(C);
\draw[dashed] (C)--(D); 
\draw[dashed] (D)--(A);
\end{scope}


\begin{scope}[shift={(3.5,0.6)}]
\fill[pattern=north east lines, pattern color =gray] (0,0) rectangle (1.5,1);
\node[open] (A) at (0,0) {};
\node[open] (B) at (1.5,0) {};
\node[open] (C) at (1.5,1) {};
\node[open] (D) at (0,1) {}; 

\draw[dashed] (A) -- (B);
\draw[dashed] (B)--(C);
\draw[dashed] (C)--(D); 
\draw[dashed] (D)--(A);
\end{scope}

\begin{scope}[shift={(0,0.6)}]
\fill[pattern=north east lines, pattern color =gray] (0,0) rectangle (1.5,1);
\node[open] (A) at (0,0) {};
\node[open] (B) at (1.5,0) {};
\node[open] (C) at (1.5,1) {};
\node[open] (D) at (0,1) {}; 

\draw[dashed] (A) -- (B);
\draw (B)--(C);
\draw[dashed] (C)--(D); 
\draw (D)--(A);
\end{scope}


\begin{scope}[shift={(0.02,-2.4)}]

\fill[pattern=north east lines, pattern color=gray] (1.8,4)--(1.8,3.1)--(3.185,4)--cycle;

\fill[pattern=north east lines, pattern color=gray] (1.92,3)--(3.3,3)--(3.3,3.9)--cycle;

\draw[dashed] (1.8,4)--(1.8,3.1);
\draw[dashed] (1.8,4)--(3.185,4);
\draw[dashed] (1.92,3)--(3.3,3.9);

\draw (1.8,3.1)--(3.185,4);
\draw[dashed] (1.92,3)--(3.3,3);
\draw (3.3,3)--(3.3,3.9);

\node[open] at (1.8,4) {};
\node[open] at (1.8,3.1) {};
\node[open] at (1.92,3) {};

\node[open] at (3.3,3) {};

\node[open] at (3.3,3.9) {};
\node[open] at (3.185,4) {};
    
\end{scope}

\begin{scope}[shift={(1.71,-1.2)}]

\fill[pattern=north east lines, pattern color=gray] (1.8,4)--(1.8,3.1)--(3.185,4)--cycle;

\fill[pattern=north east lines, pattern color=gray] (1.92,3)--(3.3,3)--(3.3,3.9)--cycle;

\draw[dashed] (1.8,4)--(1.8,3.1);
\draw[dashed] (1.8,4)--(3.185,4);
\draw[dashed] (1.92,3)--(3.3,3.9);

\draw (1.8,3.1)--(3.185,4);
\draw[thick] (1.92,3)--(3.3,3);
\draw[dashed] (3.3,3)--(3.3,3.9);

\node[open] at (1.8,4) {};
\node[open] at (1.8,3.1) {};
\node[open] at (1.92,3) {};

\node[open] at (3.3,3) {};

\node[open] at (3.3,3.9) {};
\node[open] at (3.185,4) {};
\node at (2.95,3.3) {};
\end{scope}
\end{tikzpicture}
    \caption{An \il{SE}-Partition of the $2$-torus.}
    \label{fig:SEdecomp}
\end{figure}

 Note that most of our partition classes look like \il{S}-partitionable classes. That is, $\partial P_1$ is \il{S}-partitionable and for the majority of the classes we have $\Gamma_{\sigma_i}=(P_i\setminus\{\sigma_i\}) \uplus \partial \Gamma_{\sigma_i}$ with $\Gamma_{\sigma_i}$ near-Eulerian and $\tilde \Sigma\Gamma_{\sigma_i}$ \il{S}-partitionable, except for two distinguished classes that have a slightly different structure:
 
\begin{figure}[!ht]
    \centering
    
\begin{tikzpicture}[scale=1.8, thick]
\tikzset{
  open/.style={circle,draw,fill=white,inner sep=1.2pt},
}
\tikzset{
  open/.style={circle,draw,fill=white,inner sep=1.2pt},
  filled/.style={circle,draw,fill=black,inner sep=1.2pt},
}
    
\fill[pattern=north east lines, pattern color =gray] (0,0) rectangle (1.5,1);
\node[open] (A) at (0,0) {};
\node[open] (B) at (1.5,0) {};
\node[open] (C) at (1.5,1) {};
\node[open] (D) at (0,1) {}; 

\draw (A) -- (B);
\draw[dashed] (B)--(C);
\draw (C)--(D); 
\draw[dashed] (D)--(A);
\end{tikzpicture}
    \caption{Partition class with a non-near Eulerian $\Gamma_{\sigma_i}$.}
\end{figure}
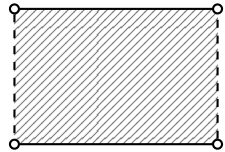
However, by removing the facet of the partition class, it can be decomposed into two disjoint components, each one coming from an \il{S}-partitionable class. In the \il{S}-partitionable proof, the vertex $v$ representing the (middle) chain $\{\sigma_i\}$ will have two components that it could be associated to. We compensate for this issue by duplicating $v$ so that the rank set $\{3\}$ is counted twice. This is demonstrated in Figure~\ref{fig:doublesemisusp}, by seeing the semi-suspension of each component.

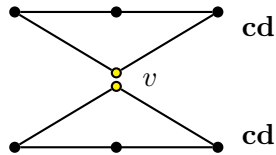
\begin{figure}[!ht]
    \centering
    
\begin{tikzpicture}[scale=1.8, thick]

\tikzset{
  open/.style={circle,draw,fill=white,inner sep=1.2pt},
  filled/.style={circle,draw,fill=black,inner sep=1.2pt},
  filledy/.style={circle,draw,fill=yellow,inner sep=1.2pt},
}
    


\node[filled] (A) at (0,0) {};
\node[filled] (B) at (1.5,0) {};
\node[filled] (C) at (1.5,1) {};
\node[filled] (D) at (0,1) {}; 
\node[filled] (E) at (0.75,0) {}; 
\node[filledy] (F) at (0.75,0.45) {}; 
\node[filledy] (G) at (0.75,0.55) {}; 
\node[filled] (H) at (0.75,1) {};

\draw (A) -- (B);
\draw (C)--(D); 
\draw (A)--(F);
\draw (B)--(F);
\draw (C)--(G);
\draw (D)--(G);

\node at (1.8,0.9) {$\g{cd}$}; 
\node at (1.8,0.1) {$\g{cd}$}; 
\node at (1,0.5) {$v$};
\end{tikzpicture}
    \caption{Decomposition into two blocks.}
    \label{fig:doublesemisusp}
\end{figure}

By the semi-Eulerian correction to the flag $f$-vector, there are exactly $\chi(\mathbb{S}^2)-\chi(T)=2$ extra elements with rank set $\{3\}$, so we can duplicate these vertices in the contribution of these two distinguished classes without issue. After making this adjustment, we illustrate the contributions of each partition class in Figure~\ref{fig:torus_partition}.

In doing so, we obtain the $\g{cd}$-index of $T$ as $\Phi(T)=\g{c}^3+7\g{dc}+13\g{cd}$ and this coincides the $\g{cd}$-index of Definition \ref{def:semicd} and Theorem \ref{th:semicd}.
\newpage 
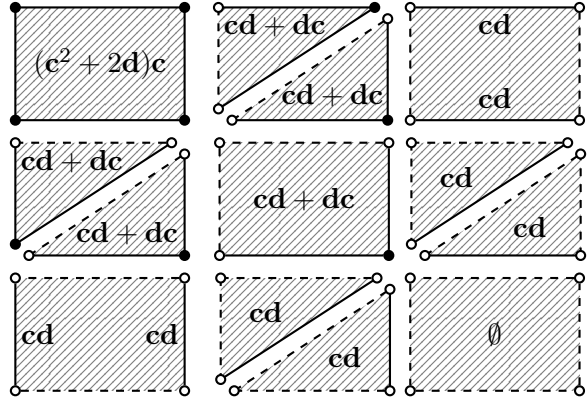
\begin{figure}[!ht]
\centering
\begin{tikzpicture}[scale=1.5, thick]
\tikzset{
  open/.style={circle,draw,fill=white,inner sep=1.2pt},
}
\tikzset{
  open/.style={circle,draw,fill=white,inner sep=1.2pt},
  filled/.style={circle,draw,fill=black,inner sep=1.2pt},
}

\draw[pattern=north east lines, pattern color=gray] (0,3) rectangle (1.5,4);
\node[filled] at (0,4) {};
\node[filled] at (0,3) {};
\node[filled] at (1.5,3) {};
\node[filled] at (1.5,4) {};
\node at (0.75,3.5) {$(\textbf{c}^2+2\textbf{d})\textbf{c}$};

\begin{scope}

\fill[pattern=north east lines, pattern color=gray] (1.8,4)--(1.8,3.1)--(3.185,4)--cycle;

\fill[pattern=north east lines, pattern color=gray] (1.92,3)--(3.3,3)--(3.3,3.9)--cycle;

\draw[dashed] (1.8,4)--(1.8,3.1);
\draw (1.8,4)--(3.185,4);
\draw[dashed] (1.92,3)--(3.3,3.9);

\draw (1.8,3.1)--(3.185,4);
\draw (1.92,3)--(3.3,3);
\draw (3.3,3)--(3.3,3.9);

\node[open] at (1.8,4) {};
\node[open] at (1.8,3.1) {};
\node[open] at (1.92,3) {};

\node[filled] at (3.3,3) {};

\node[open] at (3.3,3.9) {};
\node[filled] at (3.185,4) {};

\node at (2.3,3.84) {$\textbf{cd}+\textbf{dc}$};
\node at (2.81,3.22) {$\textbf{cd}+\textbf{dc}$};
\end{scope}


\begin{scope}[shift={(3.5,3)}]
\fill[pattern=north east lines, pattern color = gray] (0,0) rectangle (1.5,1);
\node[open] (A) at (0,0) {};
\node[open] (B) at (1.5,0) {};
\node[open] (C) at (1.5,1) {};
\node[open] (D) at (0,1) {}; 

\draw (A) -- (B);
\draw[dashed] (B)--(C);
\draw (C)--(D); 
\draw[dashed] (D)--(A);
\node at (0.75,0.85) {$\textbf{cd}$};
\node at (0.75,0.2) {$\textbf{cd}$};
\end{scope}

\begin{scope}[shift={(-1.8,-1.2)}]

\fill[pattern=north east lines, pattern color=gray] (1.8,4)--(1.8,3.1)--(3.185,4)--cycle;

\fill[pattern=north east lines, pattern color=gray] (1.92,3)--(3.3,3)--(3.3,3.9)--cycle;

\draw (1.8,4)--(1.8,3.1);
\draw[dashed] (1.8,4)--(3.185,4);
\draw[dashed] (1.92,3)--(3.3,3.9);

\draw (1.8,3.1)--(3.185,4);
\draw (1.92,3)--(3.3,3);
\draw (3.3,3)--(3.3,3.9);

\node[open] at (1.8,4) {};
\node[filled] at (1.8,3.1) {};
\node[open] at (1.92,3) {};

\node[filled] at (3.3,3) {};

\node[open] at (3.3,3.9) {};
\node[open] at (3.185,4) {};

\node at (2.3,3.84) {$\textbf{cd}+\textbf{dc}$};
\node at (2.79,3.20) {$\textbf{cd}+\textbf{dc}$};
\end{scope}

\begin{scope}[shift={(1.81,1.8)}]
\fill[pattern=north east lines, pattern color=gray] (0,0) rectangle (1.5,1);
\node[open] (A) at (0,0) {};
\node[filled] (B) at (1.5,0) {};
\node[open] (C) at (1.5,1) {};
\node[open] (D) at (0,1) {}; 

\draw (A) -- (B);
\draw (B)--(C);
\draw[dashed] (C)--(D); 
\draw[dashed] (D)--(A);

\node at (0.75,0.5) {$\textbf{cd}+\textbf{dc}$};
\end{scope}


\begin{scope}[shift={(3.5,0.6)}]
\fill[pattern=north east lines, pattern color=gray] (0,0) rectangle (1.5,1);
\node[open] (A) at (0,0) {};
\node[open] (B) at (1.5,0) {};
\node[open] (C) at (1.5,1) {};
\node[open] (D) at (0,1) {}; 

\draw[dashed] (A) -- (B);
\draw[dashed] (B)--(C);
\draw[dashed] (C)--(D); 
\draw[dashed] (D)--(A);

\node at (0.75,0.5) {$\emptyset$};
\end{scope}

\begin{scope}[shift={(0,0.6)}]
\fill[pattern=north east lines, pattern color=gray] (0,0) rectangle (1.5,1);
\node[open] (A) at (0,0) {};
\node[open] (B) at (1.5,0) {};
\node[open] (C) at (1.5,1) {};
\node[open] (D) at (0,1) {}; 

\draw[dashed] (A) -- (B);
\draw (B)--(C);
\draw[dashed] (C)--(D); 
\draw (D)--(A);

\node at (0.2,0.5) {$\textbf{cd}$};
\node at (1.3,0.5) {$\textbf{cd}$};
\end{scope}


\begin{scope}[shift={(0.02,-2.4)}]

\fill[pattern=north east lines, pattern color=gray] (1.8,4)--(1.8,3.1)--(3.185,4)--cycle;

\fill[pattern=north east lines, pattern color=gray] (1.92,3)--(3.3,3)--(3.3,3.9)--cycle;

\draw[dashed] (1.8,4)--(1.8,3.1);
\draw[dashed] (1.8,4)--(3.185,4);
\draw[dashed] (1.92,3)--(3.3,3.9);

\draw (1.8,3.1)--(3.185,4);
\draw[dashed] (1.92,3)--(3.3,3);
\draw (3.3,3)--(3.3,3.9);

\node[open] at (1.8,4) {};
\node[open] at (1.8,3.1) {};
\node[open] at (1.92,3) {};

\node[open] at (3.3,3) {};

\node[open] at (3.3,3.9) {};
\node[open] at (3.185,4) {};

\node at (2.2,3.7) {$\textbf{cd}$};
\node at (2.9,3.3) {$\textbf{cd}$};
\end{scope}

\begin{scope}[shift={(1.71,-1.2)}]

\fill[pattern=north east lines, pattern color=gray] (1.8,4)--(1.8,3.1)--(3.185,4)--cycle;

\fill[pattern=north east lines, pattern color=gray] (1.92,3)--(3.3,3)--(3.3,3.9)--cycle;

\draw[dashed] (1.8,4)--(1.8,3.1);
\draw[dashed] (1.8,4)--(3.185,4);
\draw[dashed] (1.92,3)--(3.3,3.9);

\draw (1.8,3.1)--(3.185,4);
\draw[thick] (1.92,3)--(3.3,3);
\draw[dashed] (3.3,3)--(3.3,3.9);

\node[open] at (1.8,4) {};
\node[open] at (1.8,3.1) {};
\node[open] at (1.92,3) {};

\node[open] at (3.3,3) {};

\node[open] at (3.3,3.9) {};
\node[open] at (3.185,4) {};

\node at (2.2,3.7) {$\textbf{cd}$};
\node at (2.85,3.25) {$\textbf{cd}$};
\end{scope}
\end{tikzpicture}
    \caption{Contribution of each class to $\Phi(T)$.}
    \label{fig:torus_partition}
\end{figure}

\subsection{SE-Partitionable posets}

We can now define \il{SE}-partitionability.

\begin{definition}
    Let $\Omega$ be a semi-Eulerian poset of rank $d+1$. We say that $\Omega$ is \textit{\il{SE}-partitionable} if either $\Omega=\{\emptyset\}$, or else $\Omega$ has rank at least 2 and there exists a partition of the elements of $\Omega$ into partition classes $P_\sigma \subseteq \overline{\sigma}$ for each coatom $\sigma$ of $\Omega$. Furthermore, for $\Gamma_\sigma  \coloneq \overline{P_\sigma \setminus \{\sigma\}}$ we require:
    \begin{enumerate}[(a)]
        \item There is a unique initial coatom $\sigma_1$ for which $P_{\sigma_1}=\overline{\sigma_1}$ and $P_{\sigma_1} \setminus \{\sigma_1\}$ is \il{S}-partitionable of rank $d$.
        \item For every other coatom $\sigma$, we require either: 
        \begin{itemize}
            \item $P_{\sigma}\setminus \{\sigma\}$ admits a partition into $k_\sigma$ (disjoint) sub-classes $P_\sigma^j$, where each $\Gamma_\sigma^j\coloneq\overline{P_\sigma^j}$ given as $\Gamma_{\sigma}^j = P_\sigma^j \uplus \partial \Gamma_\sigma^j$ is near-Eulerian of rank $d-1$ and $\tilde{\Sigma}\Gamma_{\sigma_i}^j$ is \il{S}-partitionable of rank $d$ such that the initial coatom is the added coatom $\tau_{\sigma}^j$.  
            \item $P_\sigma=\{\sigma\}$.
        \end{itemize}
    \end{enumerate}
    \end{definition}
We will say that the coatoms and partition classes different from the initial with $P_\sigma\ne \{\sigma\}$ are ordinary. Observe that in contrast to the definition of \il{S}-partitionable posets, each ordinary class $P_\sigma$ can contribute any number of \il{S}-partitionable posets arising from semisuspensions. This variance will encode the shift that occurs in Definition~\ref{def:semicd}. Notice that a discrete set of $N$ points $p_1,p_2,\ldots,p_N$ has an \il{SE}-partition into classes of the form $[\emptyset,p_1]$ and $\{p_i\}$ for $i\geq 2$. Meanwhile, a $2$-torus has the \il{SE}-partition provided in Figure~\ref{fig:SEdecomp}. From these two examples, we illustrate that having a $k_\sigma>1$ corresponds to adding $k_\sigma-1$ to $f_{\{d\}}(\Omega)$. Moreover, having $M$ partition classes $P_i$ such that $P_\sigma=\{\sigma\}$ corresponds to subtracting $M-1$ from $f_{\{d\}}(\Omega)$.

\subsection{SE-partitionable non-negativity} The proof is again by induction and the main difference with the proof of Theorem~\ref{Th.S-Partitionable} lies in the fact that for a middle chain $F$ of $\sigma$, the chain $\hat{F}\coloneq F \setminus \{\sigma\}$ may be contained in more than one $\partial \Gamma_\sigma^j$. However, we will compensate by introducing additional placeholders with same rank set as $F$ in our pre-blocks. This adjustment in the pre-blocks will ensure the existence of the \textbf{cd}-index, and it will coincide with the definition of the \textbf{cd}-index of $\Omega$ given in \cite{semi}. 

\begin{theorem}\label{thm:SE_part}
  The \g{cd}-index of an \il{SE}-partitionable poset can be recursively computed in terms of its \il{S}-partitionable parts of smaller rank, and is consequently non-negative.
\end{theorem}

\begin{proof} Let $\Omega$ be an \il{SE}-partitionable poset of rank $d+1$. If $d=0$, then $\Omega=\{\emptyset\}$ and has \g{cd}-index equal to $1$. If $d=1$, with coatoms $\sigma_1,\hdots,\sigma_t$, the only possible \il{SE}-partition is given by $P_{\sigma_1}=\{\hat0,\sigma_1\},\{\sigma_2\},\hdots,\{\sigma_t\}$. Taking into account the contribution of $\g{c}$ coming from the \il{S}-partitionable class $\partial \overline{\sigma_1} = \{\emptyset\}$, together with empty contributions for the partition classes of size $1$, we obtain $\Phi(\Omega)=\Phi(\partial\overline{\sigma_1})\cdot \g{c}=\g{c}$. Since the semi-Eulerian \g{cd}-index of $\Omega$ is obtained by removing $\chi(\Omega)-\chi(\mathbb{S}^0)$ chains of length $1$ containing (single) coatoms of $\Omega$ we acquire the desired $\g{cd}$-index. Assume $d\geq 2$ and recall that properties (P1) and (P2) hold for all \il{S}-partitionable posets.
\medskip
\newline 
    \textbf{Size 3 pre-blocks:} As in the proof of Theorem~\ref{Th.S-Partitionable}, use the \il{SE}-partition to define the following pre-block of size $3$ for every $G\in \Delta(\Omega)$ with $d\notin r(G)$, 
    \newline 
    $$\{\beta(G),G,r(G)\cup\{d\}\}_{G:d\notin r (G)}.$$
    \medskip
    \newline
    \textbf{Size 4 pre-blocks:} Let $F$ be an element of $\Delta(\Omega)$ that satisfies $d\in r(F)$, and upon removing $\sigma$ from the chain $F$, to create the chain $\hat{F}$, we have that $\hat{F}\in \Delta(\partial\Gamma_\sigma^j)$ for some ordinary coatom $\sigma$. We call such chains \textit{middle chains of $\sigma$}. Then, for every middle chain $F$ of $\sigma$, and each $j$ such that $\hat{F} \in \Delta(\Gamma_\sigma^j$), use the \textit{S}-partition of $\tilde \Sigma \Gamma_\sigma^j$ to consider the top chain $\tilde{F}_j$ of $\hat{F}$ in $\Sigma \tilde{\Gamma}_\sigma^{j}$, obtained by adjoining a cocoatom of $\Omega$ within $P_\sigma^{j} \subseteq P_\sigma$. Insert a placeholder with rank set $r(F)$ into the pre-block of size $3$ that contains $\tilde{F}_j$ to create a pre-block of size $4$ of the form:$$\{r(F),\beta(G),G=\tilde{F}_j,r(G)\cup \{d\}\}.$$
    By property (P1), this assignment is well-defined.
    Finally, these blocks indeed have size four, as no two middle chains are placed in the same pre-block of size three. Indeed, this is guaranteed by the fact that $\tilde{F}_j$ lies in $P_\sigma^j$, and the fact that the collection of all $P_\sigma^{j}$ forms a partition of $P_\sigma\setminus \{\sigma\}$.
\medskip
\newline 
\textbf{Placeholders:} We will redeem the placeholders in a similar way as the proof of Theorem~\ref{Th.S-Partitionable}. Let $\mathcal{B}$ denote the collection of chains (with placeholders) in the constructed pre-blocks. By construction, we have
$$
f_K(\mathcal{B}) = f_K(\Omega) = f'_K(\Omega),$$
for all $K \subseteq [d-1]$. Since the modified chain polynomial of $\Omega$ satisfies the generalized Dehn--Sommerville equations and $\Omega$ has a \g{cd}-index, it suffices to show that $\mathcal{B}$ has a \g{cd}-index. In that case, it must be that $f_K(\mathcal{B})=f'_K(\Omega)$ for every $K\subseteq [d]$ because of the generalized Dehn--Sommerville equations.
\medskip
\newline 
\textbf{The \textbf{cd}-index:} Fix a coatom $\sigma$. We shall count the contributions of $\sigma$ from its pre-blocks as in the \il{S}-partitionable case. Notice that $\sigma$ will have no associated pre-blocks if $P_\sigma=\{\sigma\}$; it will have an empty contribution to the \g{cd}-index of $\Omega$ and we consider $k_\sigma$ equal to $0$. Consider all pre-blocks of size four of the form $\{r(F),\beta(G),G,r(G) \cup\{d\}\}$ with $\sigma \in \beta(G)$. Then, it contributes $S_h(\hat{F})(\g{ab}+\g{ba})$ to the \g{a}\g{b}-polynomial of $\Omega$ as in the proof of Theorem~\ref{Th.S-Partitionable}. Furthermore, because there is a bijection (given by adding/deleting $d$ from the rank set) between placeholders corresponding to middle chains of $\sigma$ and the disjoint union $\uplus_{j=1}^{k_\sigma}\Delta(\partial\Gamma_\sigma^j)$, we conclude that all of the pre-blocks of size 4 that $\sigma$ contributes yield $\sum_{j=1}^{k_\sigma}\Phi(\partial \Gamma_\sigma^{j})\cdot \mathbf{d}$. Therefore, the collection of all pre-blocks of size $4$ gives a contribution of 
$$\sum_{\substack{\sigma \text{ ordinary\ }\\ \text{in $\Omega$} }}\sum_{j=1}^{k_\sigma}\Phi(\partial \Gamma_\sigma^{j})\cdot \mathbf{d}$$
to the \textbf{cd}-index of $\Omega$. We reiterate that by Theorem $\ref{Th.S-Partitionable}$ each summand is a non-negative sum of \textbf{cd}-words. Therefore, the set of all pre-blocks of size four contribute a non-negative collection of \textbf{cd}-words to the \textbf{cd}-index of $\Omega$. It remains to count the contributions of the pre-blocks of size 3. Consider $G$ with $d\not \in r(G)$ that is in a pre-block of size 3. Then, the pre-block of size three given by $\{\beta(G),G,r(G)\cup \{d\}\}$ contributes $T_h(G)(\g{a}+\g{b})$ to the \g{ab}-polynomial of $\Omega$ as in the proof of Theorem~\ref{Th.S-Partitionable}.
We shall break into two cases. First, we sum over all chains $G$ such that the initial coatom $\sigma_1\in \beta(G)$. From these, we acquire a non-negative contribution of $$\Phi(\partial \overline{\sigma_1})\cdot \g{c}$$ because $\partial \overline{\sigma_1}$ is \il{S}-partitionable of rank $d$. It remains to count the contributions of the remaining pre-blocks of size 3. We begin by investigating which chains $G$ with $d\not \in r(G)$ are contained in these remaining pre-blocks. Such chains have the following properties:
\begin{itemize}
    \item The highest ranked element of $G$ lies in some partition class $P_\sigma$ for some ordinary coatom.
    \item The chain $G$ is not the top chain of some $\hat{F}\in \Delta(\partial \Gamma_\sigma^{j})$ in $\tilde{\Sigma}\Gamma_\sigma^{j}$. 
\end{itemize}
These remaining chains $G$ are exactly the contributions of the ordinary \il{S}-partition classes in the recursive \il{S}-partition of each $\tilde{\Sigma}\Gamma_\sigma^{j}$. By property (P2) of Theorem~\ref{Th.S-Partitionable}, we find that each such contribution to the \textbf{cd}-index of $\tilde{\Sigma}\Gamma_\sigma$ is non-negative. Let $\mathcal{C}_{\Omega}(\sigma)$ denote the contribution of $\sigma$ to the \textbf{cd}-index of $\Omega$. We get the following recursive formula for the \g{cd}-index of $\Omega$.
$$\Phi(\Omega)=\Phi(\partial \overline{\sigma_1})\cdot \g{c} + \sum_{\substack{\sigma \text{ ordinary\ }\\ \text{in $\Omega$} }} \sum_{j=1}^{k_\sigma}\Phi(\partial \Gamma_\sigma^{j})\cdot \mathbf{d} + \sum_{\substack{\sigma \text{ ordinary\ }\\ \text{in $\Omega$} }}\sum_{j=1}^{k_\sigma}\sum_{\substack{\ \omega \text{ ordinary}\\ \text{in $\tilde\Sigma\Gamma_\sigma^j$} }}\mathcal{C}_{\tilde{\Sigma}\Gamma_{\sigma_i}^{j}}(\omega) \cdot \g{c}.$$
\end{proof}\quad

\subsection{Consequences}

\begin{example} Consider the \il{SE}-partition of the $2$-torus $T$ given in Figure~\ref{fig:SE-Torus}. By adding the contribution of each partition class, we obtain the \g{cd}-index of $T$ as $\Phi(T)=\g{c}^3+11\g{dc}+9\g{cd}.$
\end{example}

\begin{figure}[!ht]
    \centering
    \begin{tikzpicture}[scale=1, thick]
    \coordinate (A) at (0,0);
    \coordinate (B) at (0,-2);
    \coordinate (C) at (1,-2);
    \coordinate (D) at (2,-1);
    \coordinate (E) at (2,0);

    \coordinate (F) at (0,-2.5);
    \coordinate (G) at (0,-4.5);
    \coordinate (H) at (1,-4.5);
    \coordinate (I) at (1,-2.5);

    \coordinate (J) at (0,-5);
    \coordinate (K) at (0,-7);
    \coordinate (L) at (2,-7);
    \coordinate (M) at (2,-6);
    \coordinate (N) at (1,-5);

    \coordinate (Ñ) at (3,-6);
    \coordinate (O) at (3,-7);
    \coordinate (P) at (5,-7);
    \coordinate (Q) at (5,-6);

    \coordinate (R) at (6,-6);
    \coordinate (S) at (6,-7);
    \coordinate (T) at (8,-7);
    \coordinate (U) at (8,-5);
    \coordinate (V) at (7,-5);

    \coordinate (W) at (7,-4.5);
    \coordinate (Y) at (8,-4.5);
    \coordinate (X) at (8,-2.5);
    \coordinate (Z) at (7,-2.5);

    \coordinate (1) at (7,-2);
    \coordinate (2) at (8,-2);
    \coordinate (3) at (8,0);
    \coordinate (4) at (6,0);
    \coordinate (5) at (6,-1);

    \coordinate (6) at (3,0);
    \coordinate (7) at (3,-1);
    \coordinate (8) at (5,-1);
    \coordinate (9) at (5,0);

\foreach \p/\q in {A/B, B/C, D/E, A/E, J/N, Ñ/O, R/S, U/T,W/Y, X/Y, 1/2, 3/4, 8/9, 6/9}
  \draw (\p)--(\q);

\foreach \p/\q in {C/D, F/G, G/H, H/I, I/F, J/K, K/L, L/M, M/N, O/P, P/Q, Q/Ñ, S/T, U/V, V/R, X/Z, Z/W, 5/1, 2/3, 4/5, 7/8, 6/7}
  \draw[dashed] (\p)--(\q);
  
\foreach \p/\q/\r/\s in {F/G/H/I, Ñ/O/P/Q, W/Y/X/Z, 7/8/9/6}
   \fill[pattern=north east lines, pattern color=gray] (\p)--(\q)--(\r)--(\s)--cycle;
\foreach \p/\q/\r/\s/\t in {A/B/C/D/E, J/K/L/M/N, R/S/T/U/V, 1/2/3/4/5}
   \fill[pattern=north east lines, pattern color=gray] (\p)--(\q)--(\r)--(\s)--(\t)--cycle;

 \foreach \p in {A,B,E,Y,9}
\filldraw[fill=black, draw=black, thick] (\p) circle (0.08);
\foreach \p in {C,D,F,G,H,I,J,K,L,M,N,Ñ,O,P,Q,R,S,T,U,V,W,X,Z,1,2,3,4,5,6,7,8}
\filldraw[fill=white, draw=black, thick] (\p) circle (0.08);

\filldraw[pattern=north east lines, pattern color=gray] 
  ($(I)+(0.5,0)$)--($(H)+(0.5,0)$)--($(Ñ)+(0,0.3)$)--($(Q)+(0,0.3)$)--($(W)+(-0.5,0)$)--($(Z)+(-0.5,0)$)--($(8)+(0,-0.3)$)--($(7)+(0,-0.3)$)--cycle;    
\filldraw[fill=black] ($(I)+(0.5,0)$) circle (0.08);
\filldraw[fill=black]($(H)+(0.5,0)$) circle (0.08);
\filldraw[fill=black]($(Ñ)+(0,0.3)$) circle (0.08);
\filldraw[fill=black]($(Q)+(0,0.3)$) circle (0.08);
\filldraw[fill=black] ($(W)+(-0.5,0)$) circle (0.08);
\filldraw[fill=black]($(Z)+(-0.5,0)$) circle (0.08);
\filldraw[fill=black]($(8)+(0,-0.3)$) circle (0.08);
\filldraw[fill=black]($(7)+(0,-0.3)$) circle (0.08);

\node at (1,-0.75) {$3\g{dc}+\g{cd}$};
\node at (-0.3,-3.5) {$\emptyset$};
\node at (0.8, -6.15) {$\g{cd}$};
\node at (4, -6.5) {$\g{cd}$};
\node at (7.3, -6.15) {$2\g{cd}$};
\node at (8.8, -3.5) {$\g{dc}+\g{cd}$};
\node at (7.3, -0.75) {$2\g{cd}$};
\node at (4,-0.5) {$\g{dc}+\g{cd}$};
\node at (4,-3.5) {$(\g{c}^2+6\g{d})\g{c}$};
    \end{tikzpicture}
    \caption{An \textit{SE}-partition of the torus.}
    \label{fig:SE-Torus}
\end{figure}
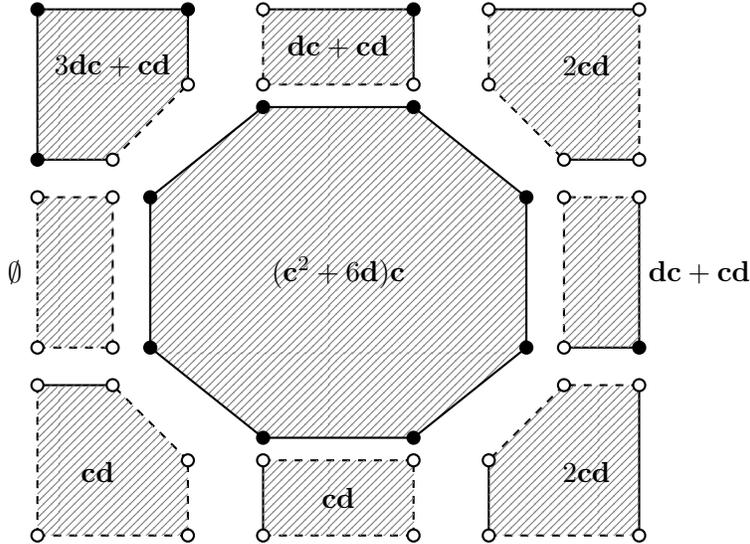

\begin{theorem}
    All semi-Eulerian partitionable simplicial complexes are \il{SE}-partitionable. 
\end{theorem}
\begin{proof} This proof follows in exactly the same way as the proof of Corollary~\ref{cor:part_implies_Spart}.
\end{proof}

\begin{proposition} Every $2$-dimensional simplicial pseudomanifold $\Delta$ is \il{SE}-partitionable
\end{proposition}
\begin{proof} Choose an initial facet $\sigma_1$ and take the initial partition class as $P_{\sigma_1}=\langle \sigma_1\rangle$.  We then build a partition inductively by adding one adjacent facet at a time.

Assume that $P_{\sigma_1},\dots,P_{\sigma_{i-1}}$ are already defined.  
Choose a facet $\sigma_i\in\Delta$ that is adjacent to the subcomplex $
\langle \sigma_1,\dots,\sigma_{i-1}\rangle,$ and consider $P_{\sigma_i}$ as the set of faces of $\sigma_i$ not already contained in a previous class:
$$P_{\sigma_i}
=\{\tau\in\langle\sigma_i\rangle :\ \tau\notin \langle\sigma_1,\dots,\sigma_{i-1}\rangle\}.
$$

The set $P_{\sigma_1}\setminus\{\sigma_1\}$ forms the boundary of a $2$-simplex, hence it is \il{S}-partitionable.  
For every subsequent class $P_{\sigma_i}$, one verifies that any upward-closed subset of the boundary of a $2$-simplex without a facet is either \il{S}-partitionable or decomposes into two subclasses satisfying the \il{SE}-conditions. Therefore, the inductively constructed partition is an \il{SE}-partition.
\end{proof}

\begin{remark}\label{rem:simp_mflds_are_SE} In the $2$-dimensional case, the above construction produces an \il{SE}-partition by taking upward-closed subsets of the faces of each facet, chosen along a spanning tree of the facet–ridge graph of the pseudomanifold $\Delta$.

In higher dimensions, such a strategy does not seem to work. Namely, one may form upward closed subsets $P_\sigma$ in this way. However, some subsets $P_\sigma$ may fail the condition that $P_\sigma$ breaks down into subclasses $P_\sigma^j$ with $\Gamma_\sigma^j=P_\sigma^j\uplus \partial \Gamma_\sigma^j$. Hence, we ask if the non-negativity of the classes that do form \il{SE}-partition classes is sufficient to guarantee that the entire \g{cd}-index is non-negative. Additionally, as proven in Juhnke-Kubitzke-Samper-Venturello~\cite{semi}, all simplicial manifolds have a non-negative \g{cd}-index.
\end{remark} 

\begin{proposition} If $\Omega, \Gamma$ are two \il{S}-partitionable posets of rank $3$ with \il{S}-partitions $\{P_\sigma\}_{\sigma\in \Omega}$ and $\{Q_\tau\}_{\tau \in \Gamma}$ respectively, then the product $\Omega\times \Gamma$ is \il{SE}-partitionable with partition classes $P_\sigma\times Q_\tau$.
\end{proposition}
\begin{proof} Consider the partition $\{P_\sigma \times Q_\tau \}$ of the faces of $\Omega \times \Gamma$. First observe that $P_{\sigma_1}\times Q_{\tau_1}\setminus \{\sigma_1\times \tau_1\}$ is combinatorially equivalent to the face poset of the boundary of a square. Hence consider
$P_{\sigma_1}\times Q_{\tau_1}$ as the initial partition class of $\Omega\times \Gamma$. 

For any other coatom $\sigma \times \tau$, note that since $\Omega$ and $\Gamma$ are both Eulerian of rank $3$, 
each $P_\sigma$ and each $Q_\tau$ is combinatorially equivalent to an edge together with some (possibly none) of its vertices. Then, one can check by hand that the product $P_\sigma \times Q_\tau$ satisfies the \il{SE}-criteria in the definition of \il{SE}-partitionability. Thus, the collection $\{P_\sigma \times Q_\tau\}$ defines an \il{SE}-partition of 
$\Omega \times \Gamma$.
\end{proof}

In general, it is not true that the product partition of two \il{SE}-partitions yields an \il{SE}-partition, although we cannot rule out that an \il{SE}-partition does exist. For instance, using $\Delta_d$ to denote the boundary of the $(d+1)$-simplex, in the product $\Delta_1 \times \Delta_2$ or $\Delta_1 \times \Delta_1 \times \Delta_1$, one obtains classes that are not \il{SE}-partition classes.

\begin{remark}
    One consequence of the proof of Theorem~\ref{thm:SE_part} is that one can obtain a \g{cd}-index for some posets that are not semi-Eulerian by `manually forcing' $h_{[d]\setminus K}$ to equal $h_K$. For example, consider the pure $2$-dimensional simplicial complex $\Omega$ on $11$ vertices in Figure~\ref{fig:non_semi}.
 The upward ideal generated by $v$ shows that $\Omega$ is not semi-Eulerian. Yet, by applying the method provided in the proof of Theorem~\ref{thm:SE_part}, one can compute a \g{cd}-index of $\Omega$ to be $\g{c}^3+9\g{dc}+10\g{cd}$. However, in general, we do not have a description of how the flag $f$-vector should change when one replaces $h_{[d]\setminus K}$ with $h'_{[d]\setminus K}\coloneq h_K$ for $K\subseteq [d-1]$. 
\end{remark}

\begin{figure}[!h]
\centering 
    \begin{tikzpicture}[scale=1.49, thick]

\coordinate (A) at (0,0);
\coordinate (B) at (1,0);
\coordinate (C) at (0.5,-0.8);
\coordinate (D) at (1.5,-0.8);

\coordinate (E) at (1.5,-1.6);
\coordinate (F) at (2.5,-0.8);

\coordinate (G) at (2,0);
\coordinate (H) at (1.5,0.8);
\coordinate (I) at (2.5,0.8);
\coordinate (J) at (1.5,1.6);
\coordinate (K) at (0.5,0.8);

\foreach \p/\q/\r in {A/B/C, B/C/D, C/D/E, E/D/F, D/F/G, G/H/I, I/H/J, K/H/J, A/B/K, B/H/K}
   \fill[pattern=north east lines, pattern color=gray] (\p)--(\q)--(\r)--cycle;

\foreach \p/\q in {A/B, A/C, B/C, C/D, D/B,  E/C, E/D, E/F, D/F, D/G, G/F, G/I, G/H, I/H, I/J, J/K, H/J, K/H, K/A, B/K, B/H}
  \draw (\p)--(\q);

 \foreach \p in {A,B,C,D,E,F,H,I,J,K}
\filldraw[fill=black, draw=black, thick] (\p) circle (0.04);
\filldraw[fill=yellow, thick] (G) circle (0.05);
\node[] at ($(G)+(0.3,0)$) {$v$};
\end{tikzpicture}
\caption{A simplicial complex that is not semi-Eulerian.}\label{fig:non_semi}
\end{figure}
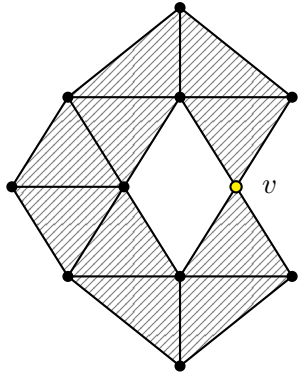

\section{Future Directions}\label{sec:future}
We collect questions that arose from this work in hopes of further developments. Our results provide new tools to prove the non-negativity of \g{cd}-indices by combinatorial means. Broadly speaking, we would like new methods to show \il{S}-partitionability and \il{SE}-partitionability of different types of posets. It is not clear if our results imply Karu's theorem, and we propose the following question.

\begin{question}
Is there an example of a Gorenstein* poset that is neither \il{S}-partitionable nor \il{SE}-partitionable? 
\end{question}

 Duval, Goeckner, Klivans, and Martin~\cite{Duval_2016} gave examples of non-partitionable Cohen-Macaulay simplicial complexes. However, their construction uses some gluings that guarantee the resulting complex is not Eulerian. Therefore, the question for Gorenstein* complexes remains open yet is likely to fail. Notice, however, that \il{SE}-partitionability may provide a slightly more flexible framework for such a conjecture, given that it is somewhat more general. 

Beyond the existence of \il{S} or \il{SE}-partitions, one may also be inclined to ask further questions about the structure of the set of all partitions on a given object. As noted in Remark~\ref{rem:rev_part}, one can make our proof of the non-negativity of the \g{cd}-index for \il{S}-partitionable posets entirely explicit, in certain cases, by looking at \il{reverse partition classes}.

\begin{question}
    Which \textit{S}-partitionable posets are reversibly partitionable?
\end{question}

When looking at partitionable simplicial complexes, the question of when a reversible partition exists also remains unclear.

\begin{question}
    Given a partitionable simplicial complex $\Omega$ with restriction faces $\{R_i\}_{i\in [t]}$ and facets $\{\sigma_i\}_{i\in [t]}$, when does the set $\{\sigma_i \setminus R_i\}_{i\in [t]}$ also form a set of restriction faces of a partition of $\Omega$?
\end{question}

A necessary condition for the answer above to  positive is that the $h$-vector of the complex has to be symmetric. Recall that any shelling order $\sigma_1,\sigma_2,\ldots,\sigma_t$ of a simplicial sphere has the property that $\sigma_t,\sigma_{t-1},\ldots,\sigma_1$ is also a shelling order. Consequently, the reversed shelling order induces reversed partitioning. Therefore, if a non-reversibly partitionable simplicial sphere exists, then it must not be shellable.

Another natural direction to explore is how the \il{S}-partitionability of  a poset $\Omega$ relates to the (simplicial) partitionability of its order complex. 

\begin{question}
  Is it true that the order complex $\Delta(\Omega)$ of every \il{S}-partitionable poset $\Omega$ is (simplicially) partitionable?
\end{question}

When $\Delta$ is a partitionable simplicial complex there is a natural combinatorial interpretation of the $h$-vectors. Namely, $h_i(\Delta)$ is the number of restriction faces of size $i$. We ask if a similar phenomenon holds for flag $h$-vectors.

\begin{question}
    For any non-negative $\g{cd}$-index, one can quickly recover a non-negative flag $h$-vector by replacing $\g{c}$ with $\g{a+b}$ and $\g{d}$ with $\g{ab+ba}$. Can one give a combinatorial answer as to why the flag $h$-vector is non-negative for \il{S}-partitionable posets?
\end{question}

Recall that all simplicial manifolds have a non-negative \g{cd}-index from Juhnke-Kubitzke, Samper, and Venturello~\cite{semi}. Therefore, we ask whether the following generalization could be true. 

\begin{question}
    Do all semi-Eulerian simplicial pseudomanifolds have a non-negative \g{cd}-index?
\end{question}

Given the combinatorial flavor of \il{S} and \il{SE}-partitionability, we hope that such an approach can shed light onto this issue. 
As noted after Remark~\ref{rem:simp_mflds_are_SE}, the guiding question is what one can say when a given partition of a poset consists mostly of \il{S} or \il{SE}-partition classes. 

More concretely, for pseudomanifolds, we suspect that we can always form partition classes in which  most $\Gamma_\sigma$ satisfy \il{SE}-partitionability criteria. The fact that each $\Gamma_\sigma$ is near-Eulerian still allows us to recursively count the \g{cd}-index of a slightly modified complex, which in turn helps us estimate the (not necessarily non-negative) contribution of the non \il{SE}-partitionable classes. Thus, the main question becomes understanding what these non \il{SE}-partitionable classes contribute and how their effect relates to the non-negative contribution of the \il{SE}-partitionable ones.

\begin{example}
  Consider the polyhedral complex $\Delta_1 \times \Delta_2$, where $\Delta_1$ and $\Delta_2$ are the boundaries of the simplicies $\langle xyz\rangle$ and $\langle abcd\rangle$, respectively and take the following facet order: 

\begin{figure}[htbp]
    \centering
    \begin{tabular}{c|c|c|c}
 $\times$ & $ xy $ & $ xz$& $ yz$  \\ \hline
$ abc$ & 1 & 2 & 7 \\ \hline 
$ abd$ & 8 &3 &4 \\ \hline 
$ acd$ & 9 &5 &11 \\ \hline 
$ bcd$ &10 &6 &12 \\
\end{tabular}
    \caption{Linear order of the facets in $\Delta_1\times \Delta_2$.}
    \label{fig:placeholder}
\end{figure}

 Partitioning the face poset assigning every face to the part of the earliest facet in the total order containing it yields a partition in which all classes satisfy the \il{SE}-partitionability hypothesis, except for those faces contained in the partition corresponding to the fourth facet. We can triangulate this facet together with the boundary of its adjacent facets and obtain a poset homeomorphic to $\mathbb{S}^1 \times \mathbb{S}^2$ in which all partition classes are \il{SE}-partitionable except for two tetrahedra, that are pictured below.
\tdplotsetmaincoords{70}{61}
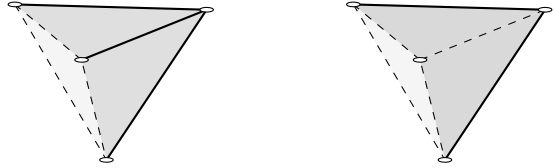
\begin{figure}[htbp]
\centering

\begin{tikzpicture}
[tdplot_main_coords,scale=3]
\begin{scope}
\coordinate (A) at (0,0.05,0.03);
\coordinate (B) at (0.7,0,0);
\coordinate (C) at (0.4,0.8,0);

\coordinate (D) at (0.35,0.32,-0.64);


\fill[gray!25] (A)--(B)--(C)--cycle;
\fill[gray!25] (B)--(C)--(D)--cycle;
\fill[gray!15,opacity=0.5] (D)--(A)--(B)--cycle;


\draw[dashed] (B)--(A);
\draw[thick] (A)--(C)--(B);
\draw[thick] (D)--(C);
\draw[dashed] (A)--(D)--(B);
\filldraw[fill=white] (C) circle (0.03);
\filldraw[fill=white] (A) circle (0.03);
\filldraw[fill=white] (B) circle (0.03);
\filldraw[fill=white] (D) circle (0.03);
\end{scope}
\begin{scope}[xshift=1.5cm]]

\coordinate (A) at (0,0.05,0.03);
\coordinate (B) at (0.7,0,0);
\coordinate (C) at (0.4,0.8,0);

\coordinate (D) at (0.35,0.32,-0.64);


\fill[gray!30] (A)--(B)--(C)--cycle;
\fill[gray!30] (B)--(C)--(D)--cycle;
\fill[gray!15,opacity=0.5] (D)--(A)--(B)--cycle;


\draw[dashed] (B)--(A);
\draw[dashed] (C)--(B);
\draw[thick] (A)--(C);
\draw[thick] (D)--(C);
\draw[dashed] (A)--(D)--(B);
\filldraw[fill=white] (C) circle (0.03);
\filldraw[fill=white] (A) circle (0.03);
\filldraw[fill=white] (B) circle (0.03);
\filldraw[fill=white] (D) circle (0.03);

\end{scope}

\end{tikzpicture}

\caption{Non \il{SE}-partitionable classes for $3$-dimensional simplicial pseudomanifolds.}
\end{figure}

The $\g{cd}$-index of this poset is given by $\g{c}^4+16\g{c}^2\g{d}+23\g{cdc}+10\g{dc}^2+34\g{d}^2$ and the contribution of the \il{SE}-partitionable parts contribute a $\g{cd}$-index equal to $\g{c}^4+17\g{c}^2\g{d}+22\g{cdc}+12\g{d}\g{c}^2+30\g{d}^2$. Subtracting these two we get that the contribution of the two faces together has to be
$-\g{c}^2\g{d}+\g{cdc}-2\g{d}\g{c}^2+4\g{d}^2$. 

\end{example}

For triangulated pseudomanifolds our expectation is that one can always find a partition in which most facets correspond to \il{SE}-partitionable classes and the few remaining facets behave like the non \il{SE}-partitionable classes above. The question is if these facets can subtract a bit from some $\g{cd}$-coefficients, but not enough to turn the contribution of the \il{SE}-partitionable part into something negative. We suspect that the number of bad facets should be somehow related to the Betti numbers of the underlying pseudomanifold and could potentially be controlled using discrete Morse theory.

Finally, we remark that there are a number of different generalizations of partitionability available in the literature and it would be interesting to know more about how they are related to this work. For example, the notion of \textit{signability} was introduced by Kleinschmidt-Onn~\cite{Kleinschmidt_1996} and in their work, \textit{totally signable} simplicial complexes were shown to satisfy McMullen's Upper Bound Theorem. Immediately after this work, Onn~\cite{ONN_1997} introduced the notion of \textit{strongly signable} posets as a generalization of \textit{dual CL-shellable} posets of Björner-Wachs~\cite{Bjrner1983}, and established the non-negativity of their flag $h$-vectors. These versions of signability vary slightly from \il{S}-partitionability and it is not clear to us if they generalize Stanley's \il{S}-shellable spheres. 

\section{Acknowledgements}
The authors thank the organizers of ECCO (Encuentro Colombiano De Combinatoria) 2024 and the hospitality of the Universidad del Cauca. DG thanks Isabella Novik for numerous suggestions, extensive guidance and instructive expository comments, and Yirong Yang for fielding questions related to partitionable simplicial complexes and the partitionability conjecture.

\printbibliography

@incollection {Bayer_2021,
    AUTHOR = {Bayer, Margaret M.},
     TITLE = {The {$cd$}-index: a survey},
 BOOKTITLE = {Polytopes and discrete geometry},
    SERIES = {Contemp. Math.},
    VOLUME = {764},
     PAGES = {1--19},
 PUBLISHER = {Amer. Math. Soc., Providence, RI},
      YEAR = {2021},
       DOI = {10.1090/conm/764/15355},
       URL = {https://doi.org/10.1090/conm/764/15355},
}

@article{Lee2010,
author = {Lee, Carl},
year = {2010},
month = {11},
pages = {},
title = {Sweeping the cd-Index and the Toric h-Vector},
volume = {18},
journal = {Electronic Journal of Combinatorics},
doi = {10.37236/553}
}

@article{Bjrner1983,
  title={On lexicographically shellable posets},
  author={Anders Bj{\"o}rner and Michelle L. Wachs},
  journal={Transactions of the American Mathematical Society},
  year={1983},
  volume={277},
  pages={323-341},
  url={https://api.semanticscholar.org/CorpusID:18914555}
}

@article{Stanley1994,
author = {Stanley, Richard P.},
journal = {Mathematische Zeitschrift},
number = {3},
pages = {483-500},
title = {Flag f-vectors and the cd-index.},
url = {http://eudml.org/doc/174665},
volume = {216},
year = {1994},
}

@article{semi,
      title={The cd-index of semi-Eulerian posets}, 
      author={Martina Juhnke-Kubitzke and José Alejandro Samper and Lorenzo Venturello},
      journal={SIAM Journal on Discrete Mathematics},
      year={to appear},
      eprint={2405.05812},
      archivePrefix={arXiv},
      primaryClass={math.CO},
      url={https://arxiv.org/abs/2405.05812}, 
}

@article{Bayer1991,
author = {Bayer, Margaret M. and Klapper, Andrew},
journal = {Discrete \& computational geometry},
number = {1},
pages = {33-48},
title = {A New Index for Polytopes.},
url = {http://eudml.org/doc/131141},
volume = {6},
year = {1991}
}

@book {Stanley_ec1,
    AUTHOR = {Stanley, Richard P.},
     TITLE = {Enumerative combinatorics. {V}olume 1},
    SERIES = {Cambridge Studies in Advanced Mathematics},
    VOLUME = {49},
   EDITION = {Second},
 PUBLISHER = {Cambridge University Press, Cambridge},
      YEAR = {2012},
     PAGES = {xiv+626},
      ISBN = {978-1-107-60262-5},
   MRCLASS = {05-02 (05A15 06-02)},
  MRNUMBER = {2868112},
}

@article {Karu_2006,
    AUTHOR = {Karu, Kalle},
     TITLE = {The {$cd$}-index of fans and posets},
   JOURNAL = {Compos. Math.},
  FJOURNAL = {Compositio Mathematica},
    VOLUME = {142},
      YEAR = {2006},
    NUMBER = {3},
     PAGES = {701--718},
      ISSN = {0010-437X,1570-5846},
   MRCLASS = {14M25 (06A07 52B05)},
  MRNUMBER = {2231198},
MRREVIEWER = {Jeremy\ L.\ Martin},
       DOI = {10.1112/S0010437X06001928},
       URL = {https://doi.org/10.1112/S0010437X06001928},
}

@phdthesis{Hachimori_2000,
    author = {Hachimori, Masahiro},
    title = {Combinatorics of constructible complexes},
    school = {Tokyo University},
    year = {2000},
}

@article{Duval_2016,
   title={A non-partitionable Cohen–Macaulay simplicial complex},
   volume={299},
   ISSN={0001-8708},
   url={http://dx.doi.org/10.1016/j.aim.2016.05.011},
   DOI={10.1016/j.aim.2016.05.011},
   journal={Advances in Mathematics},
   publisher={Elsevier BV},
   author={Duval, Art M. and Goeckner, Bennet and Klivans, Caroline J. and Martin, Jeremy L.},
   year={2016},
pages={381–395} }

@incollection{Billera_1983,
    AUTHOR = {Bayer, Margaret M. and Billera, Louis J.},
     TITLE = {Counting faces and chains in polytopes and posets},
 BOOKTITLE = {Combinatorics and algebra (Boulder, CO., 1983)},
    SERIES = {Contemp. Math.},
    VOLUME = {34},
     PAGES = {207--252},
 PUBLISHER = {Amer. Math. Soc., Providence, RI},
      YEAR = {1984},
    JOURNAL = {Combinatorics and Algebra},
       DOI = {10.1090/conm/034/777703},
       URL = {https://doi.org/10.1090/conm/034/777703},
}

@article {Billera_1985,
    AUTHOR = {Bayer, Margaret M. and Billera, Louis J.},
     TITLE = {Generalized {D}ehn-{S}ommerville relations for polytopes,
              spheres and {E}ulerian partially ordered sets},
   JOURNAL = {Invent. Math.},
  FJOURNAL = {Inventiones Mathematicae},
    VOLUME = {79},
      YEAR = {1985},
    NUMBER = {1},
     PAGES = {143--157},
       DOI = {10.1007/BF01388660},
       URL = {https://doi.org/10.1007/BF01388660},
}

@book {Ziegler_1995,
    AUTHOR = {Ziegler, G\"unter M.},
     TITLE = {Lectures on polytopes},
    SERIES = {Graduate Texts in Mathematics},
    VOLUME = {152},
 PUBLISHER = {Springer-Verlag, New York},
      YEAR = {1995},
     PAGES = {x+370},
      ISBN = {0-387-94365-X},
   MRCLASS = {52Bxx},
  MRNUMBER = {1311028},
MRREVIEWER = {Margaret\ M.\ Bayer},
       DOI = {10.1007/978-1-4613-8431-1},
       URL = {https://doi.org/10.1007/978-1-4613-8431-1},
}

@article{ONN_1997,
title = {Strongly Signable and Partitionable Posets},
journal = {European Journal of Combinatorics},
volume = {18},
number = {8},
pages = {921-938},
year = {1997},
issn = {0195-6698},
doi = {https://doi.org/10.1006/eujc.1997.0141},
url = {https://www.sciencedirect.com/science/article/pii/S0195669897901414},
author = {Shmuel Onn},
}

@incollection{Stanley_1999,
  author    = {Stanley, Richard P.},
  title     = {Positivity Problems and Conjectures in Algebraic Combinatorics},
  booktitle = {Mathematics: Frontiers and Perspectives},
  editor    = {Arnold, V. I. and Atiyah, M. F. and Lax, P. D. and Mazur, B.},
  publisher = {American Mathematical Society},
  address   = {Providence, RI},
  year      = {2000},
  pages     = {313--331},
  url       ={https://math.mit.edu/~rstan/pubs/pubfiles/116.pdf}
}

@article{Rota_1964,
author={Rota, Gian-Carlo},
title={On the foundations of combinatorial theory I. Theory of M{\"o}bius Functions},
journal={Zeitschrift f{\"u}r Wahrscheinlichkeitstheorie und Verwandte Gebiete},
year={1964},
volume={2},
number={4},
pages={340-368},
issn={1432-2064},
doi={10.1007/BF00531932},
url={https://doi.org/10.1007/BF00531932}
}

@misc{Stanley_2021,
      title={Enumerative and Algebraic Combinatorics in the 1960's and 1970's}, 
      author={Richard P. Stanley},
      year={2021},
      eprint={2105.07884},
      archivePrefix={arXiv},
      primaryClass={math.HO},
      url={https://arxiv.org/abs/2105.07884}, 
}

@article{Kleinschmidt_1996,
author={Kleinschmidt, P.
and Onn, S.},
title={Signable posets and partitionable simplicial complexes},
journal={Discrete {\&} Computational Geometry},
year={1996},
month={4},
day={01},
volume={15},
number={4},
pages={443-466},
issn={1432-0444},
doi={10.1007/BF02711519},
url={https://doi.org/10.1007/BF02711519}
}

@article{ehrenborg01,
author = {Ehrenborg, Richard},
year = {2001},
month = {09},
pages = {227-236},
title = {k-Eulerian Posets},
volume = {18},
journal = {Order},
doi = {10.1023/A:1012296719116}
}

@article{Ehrenborg_2015,
   title={Euler flag enumeration of Whitney stratified spaces},
   volume={268},
   ISSN={0001-8708},
   url={http://dx.doi.org/10.1016/j.aim.2014.09.008},
   DOI={10.1016/j.aim.2014.09.008},
   journal={Advances in Mathematics},
   publisher={Elsevier BV},
   author={Ehrenborg, Richard and Goresky, Mark and Readdy, Margaret},
   year={2015},
   month=jan, pages={85–128} }

@article{Billera_1981,
title = {A proof of the sufficiency of McMullen's conditions for f-vectors of simplicial convex polytopes},
journal = {Journal of Combinatorial Theory, Series A},
volume = {31},
number = {3},
pages = {237-255},
year = {1981},
issn = {0097-3165},
doi = {https://doi.org/10.1016/0097-3165(81)90058-3},
url = {https://www.sciencedirect.com/science/article/pii/0097316581900583},
author = {Louis J Billera and Carl W Lee},
}

@article{Karu_2023,
     author = {Karu, Kalle and Xiao, Elizabeth},
     title = {On the anisotropy theorem of {Papadakis} and {Petrotou}},
     journal = {Algebraic Combinatorics},
     pages = {1313--1330},
     year = {2023},
     publisher = {The Combinatorics Consortium},
     volume = {6},
     number = {5},
     doi = {10.5802/alco.298},
     language = {en},
     url = {https://alco.centre-mersenne.org/articles/10.5802/alco.298/}
}

@misc{papadakis_2020,
      title={The characteristic 2 anisotropicity of simplicial spheres}, 
      author={Stavros Argyrios Papadakis and Vasiliki Petrotou},
      year={2020},
      eprint={2012.09815},
      archivePrefix={arXiv},
      primaryClass={math.AC},
      url={https://arxiv.org/abs/2012.09815}
}

@misc{adiprasito_2019,
      title={Combinatorial Lefschetz theorems beyond positivity}, 
      author={Karim Adiprasito},
      year={2019},
      eprint={1812.10454},
      archivePrefix={arXiv},
      primaryClass={math.CO},
      url={https://arxiv.org/abs/1812.10454}, 
}

@book{Lundell_1969,
author = {Lundell, Albert T. and Weingram, Stephen},
address = {New York},
language = {eng},
lccn = {68026689},
publisher = {Van Nostrand Reinhold Co.},
series = {The University series in higher mathematics},
title = {The topology of CW complexes},
url = {http://www.gbv.de/dms/hbz/toc/ht000701031.pdf},
year = {1969},
}
\end{document}